\documentclass[11pt]{amsart}

\usepackage{amsmath,amsthm,amsfonts,amssymb,amscd,euscript,enumerate}
\usepackage{tikz}
\usepackage{tikz-cd}
\usepackage[all,arc]{xy}
\usepackage{enumerate}
\usepackage{mathrsfs} 
\usepackage{thmtools}
\usepackage{datetime}
\usepackage[colorlinks, citecolor = darkspringgreen, linkcolor = melon]{hyperref}
\usepackage{graphicx}
\usepackage{color}
\usepackage{cleveref}
\usepackage{comment}
\usepackage{float}
\input{xy}
\xyoption{all}
\usetikzlibrary{shapes}
\usetikzlibrary{tikzmark}
\usepackage{mathtools}
\usepackage{listofitems}
\usepackage{wrapfig}
\usepackage{caption}
\usepackage{sidecap}
\usepackage{stmaryrd}

\definecolor{melon}{rgb}{0.4, 0.2, 1}
\definecolor{darkspringgreen}{rgb}{0.14, 0.7, 0.3}
\definecolor{amethyst}{rgb}{0.6, 0.4, 0.8}

\calclayout

\setboolean{@twoside}{false}

\newtheorem{thm}{Theorem}[section]
\newtheorem{cor}[thm]{Corollary}
\newtheorem{prop}[thm]{Proposition}
\newtheorem{lem}[thm]{Lemma}

\theoremstyle{definition}
\newtheorem{defn}[thm]{Definition}

\newtheorem{exmp}[thm]{Example}

\theoremstyle{remark}
\newtheorem{rem}[thm]{Remark}

\makeatletter
\let\c@equation\c@thm
\makeatother
\numberwithin{equation}{section}

\makeatletter
\newcommand*\bigcdot{\mathpalette\bigcdot@{.5}}
\newcommand*\bigcdot@[2]{\mathbin{\vcenter{\hbox{\scalebox{#2}{$\m@th#1\bullet$}}}}}
\makeatother

\makeatletter
\def\subsection{\@startsection{subsection}{3}%
  \z@{.5\linespacing\@plus.7\linespacing}{.1\linespacing}%
  {\bfseries}}
\makeatother

\newcommand{\Z}{\mathbb{Z}}
\newcommand{\N}{\mathbb{N}}
\newcommand{\R}{\mathbb{R}}

\newcommand{\calA}{\mathcal{A}}

\newcommand{\calP}{\mathcal{P}}

\newcommand{\frakd}{\mathfrak{d}}
\newcommand{\frakm}{\mathfrak{m}}

\newcommand{\frakD}{\mathfrak{D}}
\renewcommand{\phi}{\varphi}
\renewcommand{\epsilon}{\varepsilon}

\DeclareMathOperator{\sh}{sh}
\DeclareMathOperator{\Supp}{Supp}

\DeclareMathOperator{\cut} { \setminus}

\DeclareMathOperator{\CG}{CG}

\DeclareMathOperator{\BL}{BL}

\usepackage{etoolbox}
\makeatletter
\patchcmd{\@maketitle}
  {\ifx\@empty\@dedicatory}
  {\ifx\@empty\@date \else {\vskip3ex \centering\footnotesize\@date\par\vskip1ex}\fi
   \ifx\@empty\@dedicatory}
  {}{}
\patchcmd{\@adminfootnotes}
  {\ifx\@empty\@date\else \@footnotetext{\@setdate}\fi}
  {}{}{}
\makeatother

\title{Scattering diagrams, tight gradings, and generalized positivity}
\author[Burcroff]{Amanda Burcroff}
\address{Department of Mathematics, Harvard University}
\email{\href{mailto:aburcroff@math.harvard.edu}{aburcroff@math.harvard.edu}}

\author[Lee]{Kyungyong Lee}
\address{Department of Mathematics, University of Alabama; and School of Mathematics, Korea Institute for Advanced Study}
\email{\href{mailto:kyungyong.lee@ua.edu}{kyungyong.lee@ua.edu}; \href{mailto:klee1@kias.re.kr}{klee1@kias.re.kr}}

\author[Mou]{Lang Mou}
\address{Department of Mathematics, University of Cologne}
\email{\href{mailto:langmou@math.uni-koeln.de}{langmou@math.uni-koeln.de}}

\begin{document}

\begin{abstract}
In 2013, Lee, Li, and Zelevinsky introduced combinatorial objects called compatible pairs to construct the greedy bases for rank-2 cluster algebras, consisting of indecomposable positive elements including the cluster monomials. Subsequently, Rupel extended this construction to the setting of generalized rank-2 cluster algebras by defining compatible gradings. 
We discover a new class of combinatorial objects which we call tight gradings. Using this, 
we give a directly computable, manifestly positive, and elementary but highly nontrivial formula describing rank-$2$ consistent scattering diagrams. This allows us to show that the coefficients of the wall-functions on a generalized cluster scattering diagram of any rank are positive, which implies the Laurent positivity for generalized cluster algebras and the strong positivity of their theta bases.
\end{abstract}

\maketitle
\setcounter{tocdepth}{1}
\tableofcontents

\section{Introduction}

\emph{Scattering diagrams} (or \emph{wall-crossing structures}) emerged from the work of Kontsevich--Soibelman \cite{KS} and Gross--Siebert \cite{GS} in their efforts to construct mirror manifolds, with both programs growing out of the Strominger--Yau--Zaslow conjecture \cite{SYZ} in mirror symmetry. Since then, this structure has also been utilized to encode enumerative geometric invariants \cite{GPS, Bou, GScan} and categorical invariants that count stable objects \cite{KSwcs, Bri17}. These two themes have notably overlapped in the \emph{cluster algebras} discovered by Fomin and Zelevinsky \cite{FZ} and subsequent studies, where the techniques of scattering diagrams are fundamental in solving problems in algebraic combinatorics \cite{GHKK, davison2018positive, davison2021strong, KY}.

Cluster algebras, originally devised as a combinatorial framework to address total positivity \cite{LusTotal, FZtotal} and (dual) canonical bases \cite{LusCan} in Lie theory, have themselves given rise to a wide range of intriguing algebraic and combinatorial questions. Among these, one of the most notable is the \emph{positivity phenomenon}, conjectured by Fomin and Zelevinsky \cite[Section 3]{FZ}. After remaining unsolved for over a decade, this positivity was finally proven by Lee and Schiffler \cite{LSpositive} for all skew-symmetric cluster algebras using an explicit rank-2 formula that sums over compatible pairs on Dyck paths \cite{LS}. This breakthrough led to the construction of the \emph{greedy basis} by Lee, Li and Zelevinsky \cite{LLZ}.

In their seminal work \cite{GHKK}, Gross, Hacking, Keel, and Kontsevich introduced ideas and tools from log Calabi--Yau mirror symmetry \cite{GHK}, including scattering diagrams, broken lines, and theta functions, into the study of cluster algebras. To a large extent, they constructed the canonical (or theta) basis, whose elements, known as theta functions, are parametrized by the integral tropical points in the Fock--Goncharov dual $\mathcal X$-cluster variety \cite{FG06,FG09}. Due to the positivity of the scattering diagram developed in \cite{GHKK}, the theta functions, which contain all cluster monomials, satisfy Laurent positivity. For the same reason, their multiplicative structure constants are also positive, a property referred to as \emph{strong positivity}. In this article, we combine and extend the methods of Lee--Schiffler \cite{LS, LSpositive} and Gross--Hacking--Keel--Kontsevich \cite{GHKK} to derive various new positivity results for generalized cluster algebras \cite{CS}.

In this paper, a scattering diagram in a real vector space is defined as a collection of codimension-one cones, referred to as \emph{walls}, each associated with a formal power series, called a \emph{wall-function}. We will first devote to understanding the rank-2 case in the context of computing ordered factorizations of commutators in the tropical vertex group \cite{KS, GPS}, which is equivalent to determining the scattering rays in rank-2 generalized cluster scattering diagrams \cite{Mou, CKM}. The wall-function $f_{(a, b)}(P_1, P_2)$ on the ray $\mathbb R_{\leq 0}(a, b)$ for any positive coprime integers $(a, b)$ is notoriously difficult to compute \cite{Rei11, RW, Rea, akagi2023explicit}, even when the initial wall-functions $P_1$ and $P_2$ are binomials of relatively low degrees. Although there are Coxeter-type symmetries and cluster-type discrete structures governing the appearance of some rays \cite{GP}, little is known about the wall-functions in a 2-dimensional sector known as the ``Badlands’’, when $\deg P_1 \cdot \deg P_2 > 4$.

In \Cref{sec: first main thm}, we present a directly computable, manifestly positive, elementary, yet highly nontrivial formula describing all wall-functions $f_{(a,b)}({P}_1,{P}_2)$. We show that each coefficient of the wall-functions enumerates a new class of combinatorial objects that we call \emph{tight gradings} on a maximal Dyck path. The \emph{maximal Dyck path} $\calP(m,n)$ is the lattice path from $(0,0)$ to $(m,n)$ that is closest to the main diagonal without crossing strictly above it.  A \emph{grading} on $\calP(m,n)$ is an assignment of a nonnegative integer value to each edge of $\calP(m,n)$.  A grading is \emph{tight} if it satisfies a certain combinatorial compatibility
condition (see Section \ref{sec: tight grading} for precise details).  Each tight grading has a weight depending on the coefficients of $P_1$ and $P_2$.  In the 
\begin{wrapfigure}{r}{0.28\textwidth}
\vspace{-7pt}
\includegraphics[width=0.28\textwidth]{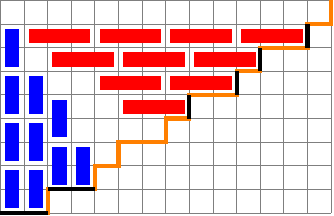}
\vspace{-20pt}
\caption{\emph{A tight grading on a maximal Dyck path.}}
\vspace{-20pt}
\label{fig: intro tight grading}
\end{wrapfigure} 
classical cluster algebra setting, this weight is either $0$ or $1$.

Pictorially, tight gradings can be represented by certain ``tilings'' by rectangles on rotations of the maximal Dyck path, as in the image to the right. The size of the first rectangle extending from each edge corresponds to its value in the grading, and edges with no rectangle extending from them have value $0$.  The relatively small space between the red and blue rectangles encodes the tightness condition, and the fact that the rectangles are disjoint encodes the compatibility condition.

\begin{thm}[Theorem \ref{first_main_thm}, \Cref{cor: bounded tight grading formula}]
In a generalized cluster scattering diagram of rank $2$, each coefficient of the wall-function $f_{(a,b)}({P}_1,{P}_2)$ is equal to the sum of weights of the corresponding tights gradings on some maximal Dyck path. 
\end{thm}

In \cite{GPS}, the coefficients in $\log f_{(a, b)}$ are proven to be interpreted by relative Gromov--Witten invariants on toric surfaces. Therefore the above theorem yields a combinatorial formula for computing these Gromov-Witten invariants in terms of tight gradings (see \Cref{cor: tight gradings GW invariants}).

Built on the rank-2 positivity demonstrated by our tight grading formula, we turn our attention to developing the positivity of higher-rank scattering diagrams towards applications in \emph{generalized cluster algebras}. These algebras, axiomatized by Chekhov and Shapiro \cite{CS} (see also \cite{Nak}), accommodate polynomial mutation rules, in contrast to binomial exchange relations introduced by Fomin and Zelevinsky \cite{FZ}.
Following \cite{GHKK}, the \emph{generalized cluster scattering diagrams} \cite{Mou, CKM} are constructed to study these algebras. Extending our rank-2 positivity and combining with work of Mou \cite[Section 8.5]{Mou} (see also \cite{CKM}), we obtain the following positivity results in all ranks. A coefficient is said to be \emph{positive} if it is a polynomial in the coefficients of the initial exchange polynomials with positive integer coefficients.

\begin{thm}[\Cref{thm: positivity higher rank}]\label{thm: intro scattering diagram positivity}
There exists a representative for (the equivalence class of) a generalized cluster scattering diagram of any rank such that the coefficients of all wall-functions are positive.
\end{thm}

\begin{cor}[\Cref{thm: positivity cluster variables}, {\cite[Conjecture 5.1]{CS}}]
In a generalized cluster algebra of any rank, the Laurent expansion of any generalized cluster variable (in an initial cluster) has positive coefficients.
\end{cor}

\begin{cor}[\Cref{thm: strong positivity}]
The theta functions defined in a generalized cluster scattering diagram of any rank have strong positivity, that is, their multiplicative structure constants are positive.
\end{cor}

In Section \ref{sec: tight grading}, we define tight gradings.  We state our first main theorem, which gives an explicit formula for wall-function coefficients in terms of tight gradings, in Section \ref{sec: first main thm}. Section \ref{sec: sd} contains preliminaries on scattering diagrams, focusing on the generalized cluster case.  Our second main theorem on the positivity of generalized cluster scattering diagrams is presented in Section \ref{sec: second main theorem} with a proof outlined.  We then describe broken lines and theta functions and deduce the positivity results for generalized cluster algebras in Section \ref{sec: broken lines and theta functions}.  In Section \ref{sec: greedy basis}, we construct the greedy basis for generalized rank-$2$ cluster algebras. We sketch the proof of the first main theorem in Section \ref{sec: proof first main theorem}.  Detailed proofs of the results in this announcement will soon be presented in a forthcoming work.

\subsection*{Acknowledgements}
AB was supported by the NSF GRFP and the Jack Kent Cooke Foundation. KL was supported by the University of Alabama, Korea Institute for Advanced Study, and the NSF grants DMS 2042786 and DMS 2302620. LM was supported by the Deutsche Forschungsgemeinschaft (DFG, German Research Foundation) -- Project-ID 281071066 -- TRR 191. The authors thank Lauren Williams for helpful suggestions. LM thanks Fang Li, Siyang Liu, Jie Pan and Michael Shapiro for helpful discussions.

\section{Tight gradings}\label{sec: tight grading}

In this section, we introduce combinatorial objects called \emph{tight gradings} that are central to our main results.

\subsection{Maximal Dyck paths}
Fix $m,n \in \Z_{\geq 0}$. Consider a rectangle with vertices $(0,0)$, $(0,n)$, $(m,0)$, and $(m,n)$ with a main diagonal from $(0,0)$ to $(m,n)$.  

\begin{defn}
  A \emph{Dyck path} is a lattice path in $(\mathbb{Z}\times \mathbb{R})\cup(\mathbb{R}\times \mathbb{Z})\subset \mathbb{R}^2$ starting at $(0,0)$ and ending at $(m,n)$, proceeding by only unit north and east steps and never passing strictly above the main diagonal. Given a collection $C$ of subpaths of a Dyck path, we denote the set of east steps by $C_1$ (resp. the set of north steps by $C_2$), and the number of east steps by $|C_1|$ (resp. the number of north steps by $|C_2|$). 
  Given an edge $e$ in a Dyck path $\calP$, let $p_e$ denote the left endpoint of $e$ if $e$ is horizontal or the top endpoint of $e$ if $e$ is vertical.  For edges $e,f$ in $\calP$, let $\overrightarrow{ef}$ denote the subpath proceeding east from $p_e$ to $p_f$, continuing cyclically around $\calP$ if $p_e$ is to the east of $p_f$. 
\end{defn}

The Dyck paths from $(0,0)$ to $(m,n)$ form a partially ordered set by comparing the heights at all vertices.  The \emph{maximal Dyck path} $\calP(m,n)$ is the maximal element under this partial order. An equivalent definition may be given as follows.

\begin{defn}\label{defn: maximal Dyck path}
For nonnegative integers $m$ and $n$, the \emph{maximal Dyck path} $\calP(m,n)$ is the path proceeding by unit north and east steps from $(0,0)$ to $(m,n)$ that is closest to the main diagonal without crossing strictly above it. We label the horizontal edges from left to right by $u_1,u_2,\dots,u_m$ and the vertical edges from bottom to top by $v_1,v_2,\dots,v_n$.
\end{defn} 

\begin{exmp} In \Cref{fig: compatible grading example}, the maximal Dyck path $\calP(6,4)$ is shown in the top left and $\calP(7,4)$ is shown in the top right.
\end{exmp}

In the setting of combinatorics on words, maximal Dyck paths are also known as Christoffel words. When $m$ and $n$ are relatively prime,  the maximal Dyck path $\calP(m,n)$ corresponds to the lower Christoffel word of slope $n/m$; see \cite{BLRS} for further details.

\subsection{Compatible gradings}
Motivated by Lee--Schiffler \cite{LS}, 
Lee, Li, and Zelevinsky \cite{LLZ} introduced combinatorial objects called \emph{compatible pairs} to construct the \emph{greedy basis} for rank-$2$ cluster algebras, consisting of indecomposable positive elements including the cluster monomials. Rupel \cite{Rupgengreed, Rup2} extended this construction to the setting of \emph{generalized} rank-$2$ cluster algebras by defining \emph{compatible gradings}.

A function  from the set of edges on $\mathcal{P}(m,n)$ to $\mathbb{Z}_{\ge0}$ is called a \emph{grading}. 

\begin{defn}
Let $E_1$ (resp. $E_2$) be the set of horizontal (resp. vertical) edges on $\calP(m,n)$, and let $E=E(m,n)=E_1\cup E_2$.  
    A grading $\omega:E\longrightarrow\mathbb{Z}_{\ge0}$  is called \emph{compatible} if for every $u\in E_1$ and $v\in E_2$, there exists an edge $e$ along the subpath $\overrightarrow{uv}$ so that at least one of the following holds:
    \begin{equation}\label{0407df:comp}
      \aligned
&    e\neq v \quad \text{and} \quad |(\overrightarrow{ue})_2|=\sum_{\tilde{u}\in (\overrightarrow{ue})_1} \omega(\tilde{u});\\
 &   e\neq u \quad \text{and} \quad |(\overrightarrow{ev})_1|=\sum_{\tilde{v}\in (\overrightarrow{ev})_2} \omega(\tilde{v}).
    \endaligned
    \end{equation}

\end{defn} 

\begin{exmp}\label{main_exmp}
For each $i\in\{1,2,3\}$, let $\omega_i: E(i+5,4)\longrightarrow \mathbb{Z}_{\ge0}$  be the grading given by $\omega_i(u_1)=\omega_i(u_2)=2$, $\omega_i(v_3)=\omega_i(v_4)=3$, and $\omega_i(e)=0$ for every $e\in E(i+5,4)\setminus\{u_1,u_2,v_3,v_4\}$.
Then $\omega_1$ is not compatible, but $\omega_2$ and $\omega_3$ are compatible. The main difference between $\omega_1$ and $\omega_2$ is that the edge $e=u_2$ in $E(7,4)$ satisfies the second condition in (\ref{0407df:comp}), with both sides of the equation equal to $6$.
\end{exmp}  

\subsection{Shadows}
In their study of compatible pairs, Lee, Li, and Zelevinsky \cite{LLZ} introduced the notion of the ``shadow'' of a set of horizontal (or vertical) edges, which Rupel \cite{Rup2} extended to the setting of gradings. 

\begin{defn}\label{def: tight grading}
For any grading $\omega$ and for any subset $S$ of $E$,  let $\omega(S)=\sum_{e\in E} \omega(e)$. 
For a vertical edge $v \in S_2$, we define its \emph{local shadow}, denoted $\sh(v; S_2)$, to be the set of horizontal edges in the shortest subpath $\overrightarrow{uv}$ of $\calP=\calP(m,n)$ such that $|(\overrightarrow{uv})_1| = \omega(\overrightarrow{uv} \cap S_2)$.
If there is no such subpath $\overrightarrow{uv}$, then we define the local shadow to be $\calP_1$.

Let the \emph{shadow} of $S_2$ be $\sh(S_2)=\bigcup_{\nu \in S_2} \sh(\nu;S_2)$. We say that $S_2$ \emph{shadows} $S_1$ if $S_1\subseteq\sh(S_2)$. Similarly $\sh(S_1)$ is defined. 

\end{defn}

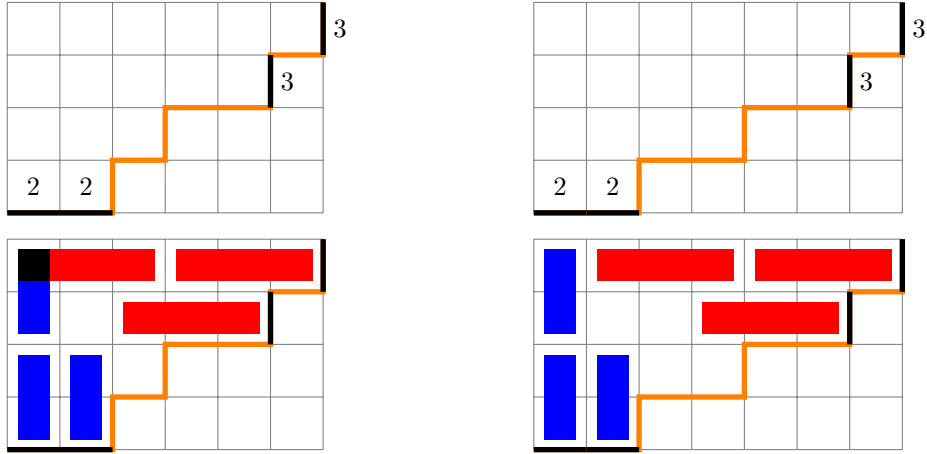
\begin{figure}[ht]
\captionsetup{width=.9\linewidth}
\begin{tikzpicture}[scale=.7]
\draw[step=1,color=gray] (0,0) grid (6,4);
\draw[line width=2,color=orange] (0,0)--(2,0)--(2,1)--(3,1)--(3,2)--(5,2)--(5,3)--(6,3)--(6,4);
\draw[line width=2pt] (0,0)--(1,0);
\draw (0.5,0.85) node[anchor=north]  {\small $2$};
\draw[line width=2pt] (1,0)--(2,0);
\draw (1.5,0.85) node[anchor=north]  {\small $2$};
\draw[line width=2pt] (6,4)--(6,3);
\draw (6,3.5) node[anchor=west]  {\small $3$};
\draw[line width=2pt] (5,3)--(5,2);
\draw (5,2.5) node[anchor=west]  {\small $3$};
%
%
\begin{scope}[shift={(10,0)}]
\draw[step=1,color=gray] (0,0) grid (7,4);
\draw[line width=2,color=orange] (0,0)--(2,0)--(2,1)--(4,1)--(4,2)--(6,2)--(6,3)--(7,3)--(7,4);
%
\draw[line width=2pt] (0,0)--(1,0);
\draw (0.5,0.85) node[anchor=north]  {\small $2$};
\draw[line width=2pt] (1,0)--(2,0);
\draw (1.5,0.85) node[anchor=north]  {\small $2$};
\draw[line width=2pt] (7,4)--(7,3);
\draw (7,3.5) node[anchor=west]  {\small $3$};
\draw[line width=2pt] (6,3)--(6,2);
\draw (6,2.5) node[anchor=west]  {\small $3$};
%
\end{scope}
\begin{scope}[shift={(0,-4.5)}]
\draw[step=1,color=gray] (0,0) grid (6,4);
\fill[blue] (0.2,0.2)--(0.2,1.8)--(0.8,1.8)--(0.8,0.2)--(0.2,0.2);
\fill[blue] (1.2,0.2)--(1.2,1.8)--(1.8,1.8)--(1.8,0.2)--(1.2,0.2);
\fill[blue] (0.2,2.2)--(0.2,2+1.8)--(0.8,2+1.8)--(0.8,2+0.2)--(0.2,2+0.2);
\fill[red] (3.2-1,3.2-1)--(3.2-1,3.8-1)--(5.8-1,3.8-1)--(5.8-1,3+0.2-1)--(3.2-1,3.2-1);
\fill[red] (3.2+1-1,3.2)--(3.2+1-1,3.8)--(5.8+1-1,3.8)--(5.8+1-1,3+0.2)--(3.2+1-1,3.2);
\fill[red] (3.2+1-3-1,3.2)--(3.2+1-3-1,3.8)--(5.8+1-3-1,3.8)--(5.8+1-3-1,3+0.2)--(3.2+1-3-1,3.2);
\fill[black] (3.2+1-3-1,3.2)--(3.2+1-3-1,3.8)--(3.8+1-3-1,3.8)--(3.8+1-3-1,3+0.2)--(3.2+1-3-1,3.2);
\draw[line width=2,color=orange] (0,0)--(2,0)--(2,1)--(3,1)--(3,2)--(5,2)--(5,3)--(6,3)--(6,4);
%
\draw[line width=2pt] (0,0)--(1,0);
\draw[line width=2pt] (1,0)--(2,0);
\draw[line width=2pt] (6,4)--(6,3);
\draw[line width=2pt] (5,3)--(5,2);
%
%
\begin{scope}[shift={(10,0)}]
\draw[step=1,color=gray] (0,0) grid (7,4);
\draw[line width=2,color=orange] (0,0)--(2,0)--(2,1)--(4,1)--(4,2)--(6,2)--(6,3)--(7,3)--(7,4);
%
\draw[line width=2pt] (0,0)--(1,0);
\draw[line width=2pt] (1,0)--(2,0);
\draw[line width=2pt] (7,4)--(7,3);
\draw[line width=2pt] (6,3)--(6,2);
%
\fill[blue] (0.2,0.2)--(0.2,1.8)--(0.8,1.8)--(0.8,0.2)--(0.2,0.2);
\fill[blue] (1.2,0.2)--(1.2,1.8)--(1.8,1.8)--(1.8,0.2)--(1.2,0.2);
\fill[blue] (0.2,2.2)--(0.2,2+1.8)--(0.8,2+1.8)--(0.8,2+0.2)--(0.2,2+0.2);
\fill[red] (3.2,3.2-1)--(3.2,3.8-1)--(5.8,3.8-1)--(5.8,3+0.2-1)--(3.2,3.2-1);
\fill[red] (3.2+1,3.2)--(3.2+1,3.8)--(5.8+1,3.8)--(5.8+1,3+0.2)--(3.2+1,3.2);
\fill[red] (3.2+1-3,3.2)--(3.2+1-3,3.8)--(5.8+1-3,3.8)--(5.8+1-3,3+0.2)--(3.2+1-3,3.2);
\end{scope}
\end{scope}
\end{tikzpicture}
\caption{In the top images, we depict gradings $\omega_1$ and $\omega_2$ on the Dyck paths $\calP(6,4)$ and $\calP(7,4)$, where edges with no weight shown are assigned weight $0$.  In the figures below, we draw blue rectangles above each horizontal edge $e$ with total height equal to the size of the local shadow of $e$, partitioned into the vertical weights contributing to the local shadow and continuing cyclically if they extend beyond the bounds of the path.  Similarly, we draw red rectangles to the left of each edge in $E_2$.  A grading is compatible if and only if the interiors of these rectangles are disjoint. Thus, the grading $\omega_1$ on $\calP(6,4)$ is not compatible, while the grading $\omega_2$ on $\calP(7,4)$ is. 
} 
\label{fig: compatible grading example}
\end{figure}

\begin{exmp}
   Consider $\omega_2$ as in  Example~\ref{main_exmp}. Let $S_1=\{u_1,u_2\}$ and $S_2=\{v_3,v_4\}$. Then $\sh(v_3;S_2)=\{u_4,u_5,u_6\}$ and $\sh(v_4;S_2)=\{u_2,u_3,\dots,u_7\}=\sh(S_2)$. Note that $\sh(S_1)=\{v_1,\dots,v_4\}$, so $S_1$ shadows $S_2$.
\end{exmp}

Partially motivated by \cite{BL}, we discovered the following definition, which is our main contribution to this paper.

\begin{defn}[tight grading] \label{main_def} Fix $\epsilon\in\{-1,1\}$. 
Fix a function $M_\epsilon:\mathbb{Z}_{>0}^2\longrightarrow \mathbb{Z}_{>0}^2$ such that if $(m,n)=M_\epsilon(\beta_1,\beta_2)$ then $$m \ge \beta_1,\quad n \ge \beta_2,\quad \text{ and }\quad \beta_1 n - \beta_2 m = \epsilon\gcd(\beta_1,\beta_2).$$
Let $(\beta_1,\beta_2)\in \mathbb{Z}_{>0}^2$. A compatible grading $\omega:E=E(M_\epsilon(\beta_1,\beta_2))\longrightarrow \mathbb{Z}_{\ge0}$ is called \emph{tight} if 
$$\omega(E_1)=\beta_2;$$
$$\omega(E_2)=\beta_1;$$
$$S_1\subseteq\sh(S_2)\text{ and }\epsilon=1,\text{ or }S_2\subseteq\sh(S_1)\text{ and }\epsilon=-1,$$
where $S_1$ is the set of horizontal edges $h$ with $\omega(h)>0$, and $S_2$ is the set of vertical edges $v$ with $\omega(v)>0$.
\end{defn}

\begin{exmp}\label{tight_exmp}
(1) The grading $\omega_2$ as in Example~\ref{main_exmp} is not tight despite $S_2\subseteq\sh(S_1)$, because $(m,n)=(7,4)$ does  not satisfy $\beta_1 n - \beta_2 m = \pm\gcd(\beta_1,\beta_2)$ for $(\beta_1,\beta_2)=(6,4)$.

\noindent (2) Let $(\beta_1,\beta_2)=(2,1)$ and $(m,n)=(3,1)$. Consider $\mathcal{P}(3,1)$.
\begin{tikzpicture}[scale=.3]
\draw[step=1,color=gray] (0,0) grid (3,1);
\draw[line width=2,color=orange] (0,0)--(3,0)--(3,1);
\fill[blue] (0.2,0.2)--(0.2,0.8)--(0.8,0.8)--(0.8,0.2)--(0.2,0.2);
\fill[red] (1.2,0.2)--(1.2,0.8)--(2.8,0.8)--(2.8,0.2)--(1.2,0.2);
\end{tikzpicture}
Suppose that $\omega(u_1)=1$, $\omega(u_2)=\omega(u_3)=0$, and $\omega(v_1)=2$. Then $\omega$ is tight. 

\noindent (3) Let $(\beta_1,\beta_2)=(4,2)$ and $(m,n)=(5,2)$. Consider $\mathcal{P}(5,2)$.
\begin{tikzpicture}[scale=.3]
\draw[step=1,color=gray] (0,0) grid (5,2);
\draw[line width=2,color=orange] (0,0)--(3,0)--(3,1)--(5,1)--(5,2);
\fill[blue] (0.2,0.2)--(0.2,0.8)--(0.8,0.8)--(0.8,0.2)--(0.2,0.2);
\fill[blue] (1+0.2,0.2)--(1+0.2,0.8)--(1+0.8,0.8)--(1+0.8,0.2)--(1+0.2,0.2);
\fill[blue] (0.2,1+0.2)--(0.2,1+0.8)--(0.8,1+0.8)--(0.8,1+0.2)--(0.2,1+0.2);
\fill[red] (2.2,0.2)--(2.2,0.8)--(2.8,0.8)--(2.8,0.2)--(2.2,0.2);
\fill[red] (2.2-1,1+0.2)--(2.2-1,1+0.8)--(2.8-1,1+0.8)--(2.8-1,1+0.2)--(2.2-1,1+0.2);
\fill[red] (2.2,1+0.2)--(2.2,1+0.8)--(4.8,1+0.8)--(4.8,1+0.2)--(2.2,1+0.2);
\end{tikzpicture}
Suppose that $\omega(u_1)=\omega(u_2)=\omega(v_1)=1$, $\omega(v_2)=3$, and $\omega(u_3)=\omega(u_4)=\omega(u_5)=0$. Then $\omega$ is tight. 

\noindent (4) Let $(\beta_1,\beta_2)=(6,3)$ and $(m,n)=(7,3)$.  Consider $\mathcal{P}(7,3)$.
\begin{tikzpicture}[scale=.3]
\draw[step=1,color=gray] (0,0) grid (7,3);
\draw[line width=2,color=orange] (0,0)--(3,0)--(3,1)--(5,1)--(5,2)--(7,2)--(7,3);
\fill[blue] (0.2,0.2)--(0.2,0.8)--(0.8,0.8)--(0.8,0.2)--(0.2,0.2);
\fill[blue] (1+0.2,0.2)--(1+0.2,0.8)--(1+0.8,0.8)--(1+0.8,0.2)--(1+0.2,0.2);
\fill[blue] (2+0.2,0.2)--(2+0.2,0.8)--(2+0.8,0.8)--(2+0.8,0.2)--(2+0.2,0.2);
\fill[blue] (0.2,1+0.2)--(0.2,1+0.8)--(0.8,1+0.8)--(0.8,1+0.2)--(0.2,1+0.2);
\fill[blue] (1+0.2,1+0.2)--(1+0.2,1+0.8)--(1+0.8,1+0.8)--(1+0.8,1+0.2)--(1+0.2,1+0.2);
\fill[blue] (0.2,2+0.2)--(0.2,2+0.8)--(0.8,2+0.8)--(0.8,2+0.2)--(0.2,2+0.2);
\fill[red] (2.2,1+0.2)--(2.2,1+0.8)--(4.8,1+0.8)--(4.8,1+0.2)--(2.2,1+0.2);
\fill[red] (2.2-1,1+1+0.2)--(2.2-1,1+1+0.8)--(4.8-1,1+1+0.8)--(4.8-1,1+1+0.2)--(2.2-1,1+1+0.2);
\fill[red] (2.2-1+3,1+1+0.2)--(2.2-1+3,1+1+0.8)--(4.8-1+3,1+1+0.8)--(4.8-1+3,1+1+0.2)--(2.2-1+3,1+1+0.2);
\end{tikzpicture}
Suppose that $\omega(v_2)=\omega(v_3)=3$, $\omega(u_1)=\omega(u_2)=\omega(u_3)=1$, and $\omega(v_1)=\omega(u_4)=\omega(u_5)=\omega(u_6)=\omega(u_7)=0$. Then $\omega$ is tight. 

\noindent (5) Let $(\beta_1,\beta_2)=(12,8)$ and $(m,n)=(14,9)$. Then the grading $\omega$ given in the first page is tight. There are total 14 tight gradings such that $\omega(h)=2$ for exactly four horizontal edges $h$, $\omega(v)=3$ for exactly four vertical edges $v$, and  $\omega(e)=0$ for all other edges  on $\mathcal{P}(14,9)$.
\end{exmp}

\begin{rem}
The word ``tight" is coined by the tight space between blue and red rectangles.
\end{rem}

\section{The first main theorem}\label{sec: first main thm}

We consider a slight variant of \emph{the tropical vertex} studied by Gross, Pandharipande, and Siebert in \cite{GPS}. Let $p_{i, j}$ be variables of degree $j$ for $i = 1, 2$ and $j\in \mathbb Z_{\geq 1}$, and $p_{1, 0} = p_{2, 0} = 1$. Fix a ground field $\Bbbk$ of characteristic zero. Let $R$ be the graded completion of the (infinitely generated) graded polynomial algebra 
\[
    A = \bigoplus_{d\geq 0} A_d = \Bbbk [p_{i, j}\mid i = 1, 2,\, j\in \mathbb Z_{\geq 1}].
\]
Let $I_k$ denote the ideal generated by elements of degree at least $k$. Let $\mathbb T = \Bbbk [x^{\pm1}, y^{\pm1}]$. Consider the complete $R$-algebra 
\[
    \mathbb T \widehat \otimes_\Bbbk R \coloneqq \lim_{\longleftarrow} \mathbb T\otimes_\Bbbk A/I_k =  \prod_{d\geq 0} \mathbb T\otimes_\Bbbk A_d.
\]

\begin{defn}[weight]
    The \emph{weight} of a grading $\omega:E(m,n)\longrightarrow\mathbb{Z}_{\ge0}$ is defined as 
    \[
        \text{wt}(\omega)=\prod_{i=1}^m p_{2,\omega(u_i)} \prod_{j=1}^n  p_{1,\omega(v_j)}\in R.
    \]
\end{defn}

\begin{exmp}\label{weight_exmp}
If $\omega$ is as in Example~\ref{tight_exmp}(2), then $\text{wt}(\omega)=p_{1,2}p_{2,1}$. If $\omega$ is as in Example~\ref{tight_exmp}(3), then $\text{wt}(\omega)=p_{1,1}p_{1,3}p_{2,1}^2$. If $\omega$ is as in Example~\ref{tight_exmp}(4), then $\text{wt}(\omega)=p_{1,3}^2 p_{2,1}^3$.
\end{exmp}

We consider a class of automorphisms in $\operatorname{Aut}_{R}(\mathbb T\widehat \otimes_\Bbbk R)$. For $(a, b)\in \mathbb Z^2$ and an element
\begin{equation}\label{eq: tropical vertex generator}
    f \in 1 + x^ay^b \prod_{d\geq 1} \mathbb \Bbbk[x^ay^b] \otimes_\Bbbk A_d,
\end{equation}
define $T_{(a, b), f}\in \operatorname{Aut}_{R}(\mathbb T\widehat \otimes_\Bbbk R)$ by
\[
    T_{(a,b), f}(x) = f^{-b}\cdot x \quad \text{and} \quad T_{(a,b), f}(y) = f^a \cdot y.
\]

Let $P_1 = \sum_{k\geq 0} p_{1, k}x^k$ and $P_2 = \sum_{k\geq 0} p_{2, k}y^k$. As observed by Kontsevich and Soibelman \cite{KS} (see also \cite[Theorem 1.3]{GPS}), there exists a unique factorization 
into an infinite ordered product
\begin{equation}\label{eq: factorization tropical vertex}
    T_{(1,0),{P}_1}T_{(0,1),{P}_2} =\prod_{b/a \text{ decreasing}} T_{(a,b),f_{(a,b)}},
\end{equation}
where the product ranges over all coprime pairs $(a,b)\in\mathbb{Z}_{\ge0}^2$ and $f_{(a, b)}$ is of the from (\ref{eq: tropical vertex generator}). The infinite product is understood as the limit of finite products when modulo each $I_k$.

Our first main theorem explicitly computes the elements $f_{(a,b)}$ in terms of tight gradings introduced in \Cref{sec: tight grading}. The outline of proof will be given in \Cref{sec: proof first main theorem}.

\begin{thm}\label{first_main_thm}
    Fix $\epsilon\in\{-1,1\}$ and a function $M_\epsilon$ as in Definition~\ref{main_def}.  
    For any coprime $(a, b)\in \mathbb Z^2_{\geq 0}$, the element $f_{(a,b)}$ in the factorization (\ref{eq: factorization tropical vertex}) is given by
    \begin{equation}\label{eq: tight grading formula}
        f_{(a,b)} = 1+\sum_{k\geq 1} \sum_\omega \emph{wt}(\omega) x^{ka} y^{kb},
    \end{equation}
    where the second sum is over all tight gradings 
    \[
        \omega:E=E(M_\epsilon(ka,kb))\longrightarrow\mathbb{Z}_{\ge0}
    \]
    with $\omega(E_1)=kb$ and $\omega(E_2)=ka$.
\end{thm}

A grading $w\colon E\rightarrow \mathbb Z_{\geq 0}$ is said to be \emph{bounded} by $(\ell_1, \ell_2)\in \mathbb N^2$ if $w\vert_{E_i}$ is bounded above by $\ell_i$ for $i = 1, 2$. Let $P_{i, \ell_i}$ be the polynomial obtained by letting $p_{i, k} = 0$ for $k>\ell_i$ in $P_i$. By the functoriality of (\ref{eq: factorization tropical vertex}), we have a direct corollary of \Cref{first_main_thm}.

\begin{cor}\label{cor: bounded tight grading formula}
    The product $T_{(1, 0), P_{1, \ell_1}}T_{(0, 1), P_{2, \ell_2}}$ uniquely factorizes into an infinite ordered product of $T_{(a, b), f_{(a, b)}}$ as in (\ref{eq: factorization tropical vertex}) where
    \[
        f_{(a,b)} = f_{(a, b)}(P_{1, \ell_1}, P_{2, \ell_2}) = 1+\sum_{k\geq 1} \sum_\omega \emph{wt}(\omega) x^{ka} y^{kb},
    \]
    and the second sum is over the tight gradings in \Cref{first_main_thm} but bounded by $(\ell_1, \ell_2)$.
\end{cor}

\begin{exmp}\label{ex: G2 f(2,1)}
    Let $(\ell_1, \ell_2) = (3, 1)$ and $(a,b)=(2,1)$. Fix $\epsilon=-1$ and let $M_\epsilon$ be such that $M_\epsilon(ka,kb)=(ka+1,kb)$ for all $k\ge1$. Then the gradings as in Example~\ref{tight_exmp}(2)(3)(4) are the only tight gradings of the form $\omega:E=E(ka+1,kb)\longrightarrow\mathbb{Z}_{\ge0}$ with $\omega(E_1)=kb$ and $\omega(E_2)=ka$ bounded by $(3, 1)$.
    Thus,  Example~\ref{weight_exmp} implies $$f_{(2,1)} = 1 + p_{1,2}p_{2,1} x^2 y + p_{1,1}p_{1,3}p_{2,1}^2 x^4 y^2 + p_{1,3}^2 p_{2,1}^3 x^6 y^3.$$
\end{exmp}

Combining with \cite{GPS}, we describe a link to Gromov--Witten theory as follows.
For any ordered partitions $\mathbf P$ and $\mathbf Q$ respectively of size $ka$ and $kb$, and of length $\ell_1$ and $\ell_2$, there is a Gromov--Witten invariant (as defined in \cite[Section 4]{GPS})
\[
    N_{a,b}[(\mathbf P, \mathbf Q)] \in \mathbb Q
\]
defined on the weighted projective plane $X_{a, b}$ under a certain relative condition with respect to the toric boundary. Take a specialization of the variables $p_{i, j}$ in $\Bbbk \llbracket s, t \rrbracket$ so that 
\[
    P_1 = (1 + sx)^{\ell_1} \quad \text{and} \quad P_2 = (1 + ty)^{\ell_2}.
\]
Now the weight $\mathrm{wt}(\omega)$ is understood with this specialization, thus a monomial in $s$ and $t$ with a positive integer scalar, which we denote by $c(\omega)s^{ka}t^{kb}$. As a direct corollary of \Cref{cor: bounded tight grading formula} and \cite[Theorem 5.4]{GPS} (which computes $\log f_{(a,b)}(P_1, P_2)$ in terms of $N_{a,b}[(\mathbf P, \mathbf Q)]$), we have

\begin{cor}\label{cor: tight gradings GW invariants}
    For fixed $\ell_1, \ell_2$ and any coprime $(a, b)\in \mathbb Z_{>0}^2$, we have
    \begin{equation*}
        \log \left( 1+\sum_{k\geq 1} \sum_{\omega} c(\omega) s^{ka}t^{kb} x^{ka}y^{kb} \right)
        = \sum_{k\geq 1} \sum_{\mathbf P, \mathbf Q} k N_{a, b}[(\mathbf P, \mathbf Q)] s^{ka}t^{kb} x^{ka}y^{kb}
    \end{equation*}
    where the second sum on the left is over all tight gradings in \Cref{cor: bounded tight grading formula} and the second sum on the right is over all ordered partitions $\mathbf P$ and $\mathbf Q$ respectively of size $ka$ and $kb$, and of length $\ell_1$ and $\ell_2$.
\end{cor}

With other specializations of $p_{i, j}$, the correspondence as above can be drawn to relate the Gromov--Witten invariants appearing in more complicated factorization formulas obtained in \cite[Theorem 5.4, 5.6]{GPS}.

Using quiver representation techniques, Reineke and Weist \cite{RW} (extending \cite{Rei11}) computed $f_{(a, b)}(P_1, P_2)$ in terms of Euler characteristics of certain moduli spaces of framed stable representations of the complete bipartite quiver with $\ell_1$ sources and $\ell_2$ sinks. Therefore through appropriate specializations of $p_{i, j}$, our tight grading formula \Cref{cor: bounded tight grading formula} provides a way to compute these Euler characteristics.

When $P_1$, $P_2$ are specialized to $P_{1, \ell_1}$, $P_{1, \ell_2}$, the factorization pattern \Cref{cor: bounded tight grading formula} plays an important role in rank-2 generalized cluster algebras. As we will see in \Cref{sec: sd}, this connection is better illustrated through equivalent structures called \emph{scattering diagrams}. Even in the simplest interesting case where $P_1=1+x^{\ell_1}$ and $P_2=1+y^{\ell_2}$, Theorem~\ref{first_main_thm} provides new formulae describing coefficients in $f_{(a, b)}$ which are essential in constructing the theta bases of rank-2 cluster algebras \cite{GHKK}.\footnote{If $\ell_1\ell_2< 4$, then it is not hard to describe $f_{(a,b)}$ (\Cref{exmp: G2}). It has been a formidable challenge to understand the case of $\ell_1\ell_2> 4$, due to its wild behavior.}

\section{Scattering diagrams}\label{sec: sd}

We start by explaining in Subsection~\ref{subsec: rank 2 sd} rank-2 scattering diagrams. Then we introduce in Subsection~\ref{subsec: higher rank sd} scattering diagrams in higher ranks for generalized cluster algebras. The algebraic setup in this section is adapted in slight difference with \Cref{sec: first main thm} to better suit the applications in cluster algebras.

\subsection{Scattering diagrams in rank 2}\label{subsec: rank 2 sd}

Fix a rank-$2$ lattice $M \cong \Z^2$ and choose a strictly convex rational cone $\sigma$ in  $M_\mathbb R \coloneqq M \otimes \mathbb R$. We take the monoid $P = \sigma \cap M$ and denote $P^+ \coloneqq P \cut \{0\}$. Set $\widehat {\Bbbk[P]}$ to be the monoid algebra $\Bbbk[P]$ completed at the maximal monomial ideal $\frakm$ generated by $\{x^m \mid m \in P^+\}$.

\begin{defn}
    A \emph{wall} is a pair $(\frakd, f_\frakd)$ consisting of a \emph{support} $\frakd \subseteq M_\R$ and a \emph{wall-function} $f_\frakd \in \widehat{\Bbbk [P]}$, where
    \begin{itemize}
        \item $\frakd$ is either a ray $\R_{\leq 0}w$ or a line $\R w$ for some $w \in P^+$;
        \item $\displaystyle f_\frakd = f_\frakd(x^w) = 1 + \sum_{k \geq 1} c_k x^{kw}$ for $c_k \in \Bbbk$.
    \end{itemize}
\end{defn}

Associated to a wall $(\frakd, f_\frakd)$ and a direction $v\in M_\mathbb R$ transversal to $\frakd$ is an algebra automorphism $\mathfrak p_{v, \frakd}\in \operatorname{Aut}(\widehat{\Bbbk[P]})$ defined by
\[
    \mathfrak p_{v, \frakd}(x^m) = x^m f_\frakd^{n(m)} \quad \text{for $m\in P$}
\]
where $n\in \operatorname{Hom}(M, \mathbb Z)$ is primitive and orthogonal to $\frakd$ in the direction $n(v)<0$. Clearly $\mathfrak p_{v, \frakd}^{-1} = \mathfrak p_{-v, \frakd}$.

\begin{defn}
A \emph{scattering diagram} $\frakD$ is a collection of walls such that the set
\[
    \mathfrak D_k \coloneqq \{(\frakd, f_\frakd) \in \frakD \mid f_\frakd \not \equiv 1 \ \operatorname{mod}\ \frakm^k\}
\]
is finite for each $k \geq 0$.
\end{defn}

A path $\gamma \colon [0, 1] \rightarrow M_\mathbb R$ is called \emph{regular} (with respect to $\mathfrak D$) if it is a smooth immersion with endpoints away from the support of any wall and only crosses walls transversally. For each $k\geq 1$, let $0<t_1<\dots<t_s<1$ be the longest sequence such that $\gamma(t_i) \in \mathfrak d_i$ for some wall $(\mathfrak d_i, f_{\mathfrak d_i})\in \mathfrak D_k$. Consider the product
\[
    \mathfrak p^{(k)}_{\gamma, \mathfrak D} = \mathfrak p_{\dot{\gamma}(t_s), \frakd_s} \circ \cdots \circ\mathfrak p_{\dot{\gamma}(t_1), \frakd_1}.
\]
We define the \emph{path-ordered product} of $\gamma$ with respect to $\mathfrak D$ to be
\[
    \mathfrak p_{\gamma, \mathfrak D} = \lim_{k\rightarrow \infty} \mathfrak p^{(k)}_{\gamma, \mathfrak D} \in \operatorname{Aut}(\widehat{\Bbbk[P]}).
\]

\begin{defn}
    A scattering diagram $\mathfrak D$ is called \emph{consistent} if for any regular path $\gamma$ the path-ordered product $\mathfrak p_{\gamma, \mathfrak D}$ depends only on the endpoints of $\gamma$, or equivalently if $\mathfrak p_{\gamma, \mathfrak D} = \mathrm{id}$ for a simple loop $\gamma$ around the origin.
\end{defn}

\begin{thm}[\cite{KS}]\label{thm: consistent sd rank 2}
    Given any (initial) scattering diagram $\mathfrak D_\mathrm{in}$ of only lines, there exists a unique consistent scattering diagram $\mathfrak D$ such that $\mathfrak D\cut \mathfrak D_\mathrm{in}$ consists of distinct rays with non-trivial wall-functions.
\end{thm}

\begin{rem}
    The initial collection of lines can be infinite. Even though we require the added rays in $\mathfrak D\cut \mathfrak D_\mathrm{in}$ to be distinct, one can overlap with an initial line. The factorization (\ref{eq: factorization tropical vertex}) is the special case of two lines.
\end{rem}

While the use of scattering diagrams originated in the study of mirror symmetry, they have since found remarkable applications in cluster algebras by the celebrated work of Gross, Hacking, Keel, and Kontsevich \cite{GHKK}. We exhibit a collection of consistent scattering diagrams devised for (generalized) cluster algebras in rank 2. Let $M = \mathbb Z^2$ and $\{e_1, e_2\}$ be the standard basis. Choose $\sigma$ to be the first quadrant of $M_\mathbb R = \mathbb R^2$. Denote $x = x^{e_1}$ and $y = x^{e_2}$. The initial scattering diagram will be two lines
\begin{equation}\label{eq: initial sd rank 2}
    \mathfrak D_\mathrm{in} = \{(\R e_1, {P}_1(x)), (\R e_2, {P}_2(y))\}
\end{equation}
where $P_i$ is a polynomial with constant term $1$ for $i = 1, 2$. There are infinitely many rays in $\mathfrak D\cut \mathfrak D_\mathrm{in}$ of the form $(\mathbb R_{\leq 0}(a, b), f_{(a, b)})$ for coprime $(a, b)\in \mathbb Z_{>0}^2$ unless $\deg P_1\cdot \deg P_2 < 4$ when there are finitely many.

\begin{exmp}\label{exmp: G2}
    We depict in \Cref{fig: G2} the case $\deg P_1 = 3$ and $\deg P_2 = 1$. The remaining finite cases can be obtained by specializing certain coefficients to zero. In \Cref{fig: G2}, the wall-functions on the added rays are
    \begin{align*}
        f_{(3,1)} & = 1 + p_{1,3}p_{2,1} x^3y, \\
        f_{(2,1)} & = 1 + p_{1,2}p_{2,1} x^2y + p_{1,1}p_{1,3}p_{2,1}^2 x^4 y^2 + p_{1,3}^2p_{2,1}^3 x^6y^3 \quad(\text{see Example}~\ref{ex: G2 f(2,1)}), \\
        f_{(3,2)} & = 1 + p_{1,3}p_{2,1}^2 x^3y^2, \\
        f_{(1,1)} & = 1 + p_{1,1}p_{2,1} xy + p_{1,2}p_{2,1}^2 x^2y^2 + p_{1,3}p_{2,1}^3 x^3 y^3.
    \end{align*}
\end{exmp}

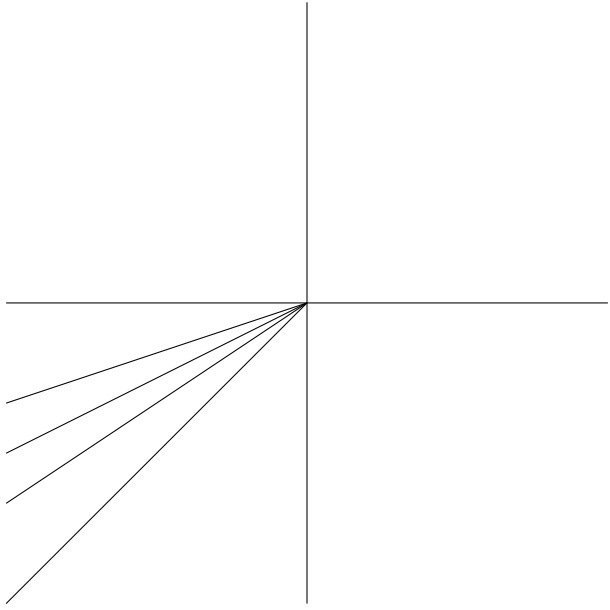
\begin{figure}
\begin{tikzpicture}[scale=.4]

\draw[] (-10,0)--(10,0);
\draw[] (0,-10)--(0,10);
\draw[] (-10, -10) -- (0, 0);
\draw[] (-10, -5) -- (0, 0);
\draw[] (-10, -20/3) -- (0, 0);
\draw[] (-10, -10/3) -- (0, 0);

\end{tikzpicture}

\caption{$P_1 = 1 + p_{1,1}x + p_{1,2}x^2 + p_{1,3}x^3$ and $P_2 = 1 + p_{2,1}y$.}
\label{fig: G2}
\end{figure}

For $P_1 = 1 + x^{\ell_1}$ and $P_2 = 1 + x^{\ell_2}$, the resulting scattering diagram $\mathfrak D_{(\ell_1, \ell_2)}$ \cite{GHKK} is famously responsible for the rank-2 cluster algebra $\mathcal A(\ell_1, \ell_2)$ \cite{FZ}. When $\ell_1 \ell_2<4$, its structure is directly derived from the example in \Cref{fig: G2} by specializing coefficients. When $\ell_1 \ell_2 \geq 4$, there is a discrete set of rays outside the closed cone spanned by
\[
    \left(-2\ell_1, -\ell_1 \ell_2 - \sqrt{\ell_1^2\ell_2^2 - 4\ell_1 \ell_2} \right)\  \text{and} \ \left(-\ell_1 \ell_2 - \sqrt{\ell_1^2\ell_2^2 - 4\ell_1 \ell_2}, -2\ell_2 \right).
\]
These rays (so-called \emph{cluster rays}) are in bijection with the cluster variables $\{x_n \mid n\in \mathbb Z,\ n\neq 0, 1, 2, 3\} \subset \mathcal A(\ell_1, \ell_2)$ such that their directions are opposite to the $d$-vectors of cluster variables. The cone itself, known as the \emph{Badlands}, has a much richer yet more elusive structure. It is known that $\frakD_{(\ell, \ell)}$ has a ray at every rational slope within the Badlands; see \cite[Section 4.7]{GP} and \cite[Example 7.10]{davison2021strong}. 
However, the wall-functions there were generally not understood, albeit closed formulas of particular slopes were proven in \cite{Rei11, RW} and some partial progress has been made in \cite{Rea, akagi2023explicit, ERS}.

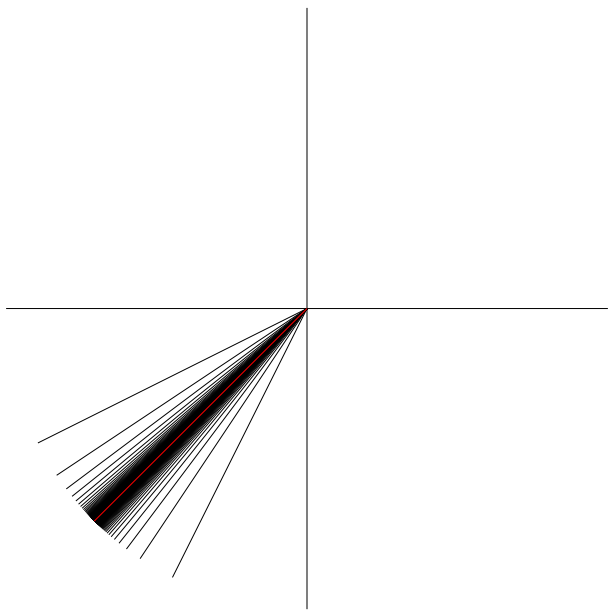
\begin{figure}
\begin{tikzpicture}[scale=.4]

\draw[line width=0.1pt] (-10,0)--(10,0);
\draw[line width=0.1pt] (0,-10)--(0,10);
\foreach \x in {2,...,200}
    {
    \draw[line width=0.1pt] (0,0)--({atan(\x/(\x-1))}:{-10 });
    \draw[line width=0.1pt] (0,0)--({atan((\x-1)/\x)}:{-10 });
    }
\draw[line width=.3pt,color=red] (0,0)--({atan(1)}:{-10 });

\end{tikzpicture}

\caption{
The walls of the scattering diagram $\frakD_{(2,2)}$ are depicted above.  The only non-cluster ray is the wall of slope $1$, shown in red. }
\label{fig: D2}
\end{figure}

\newcommand\sloperaw{2.0, 1.6, 1.5789473684, 0.4, 0.4210526316, 0.4225352113, 0.3333333333, 0.4166666667, 0.4222222222, 1.6666666667, 1.5833333333, 1.5777777778, 1.5774647887, 1.5773584906, 1.5773508595, 1.5773503116, 1.5773502722, 1.5773502694, 1.5773502692, 1.5773502692, 1.5773502692, 1.5773502692, 0.4226415094, 0.4226491405, 0.4226496884, 0.4226497278, 0.4226497306, 0.4226497308, 0.4226497308, 0.4226497308, 0.4226497308, 0.4226497308, 0.4226190476, 0.4226475279, 0.4226495726, 0.4226497195, 0.42264973,  1.5773809524, 1.5773524721, 1.5773504274, 1.5773502805, 1.57735027, 1.5773502692}

\readlist*\slope{\sloperaw}

\begin{figure}
\begin{tikzpicture}[scale=.4]
\draw[draw=none, fill=gray, opacity = 0.3] (0,0) --  ({atan(1.57735026918)}:{-10}) arc({atan(1.57735026918)}:{atan(0.4226497308)}:{-10}) -- cycle;
\draw[line width=0.1pt] (-10,0)--(10,0);
\draw[line width=0.1pt] (0,-10)--(0,10);
\draw[line width=0.1pt, color=red] (0,0)--({atan(2/3)}:{-10 });
\foreach \x in {1,...,43}   {
    \draw[line width=0.1pt] (0,0)--({atan(\slope[\x])}:{-10});
    }

\end{tikzpicture}

\caption{Suppose $P_1 = 1+x^3$ and $P_2 = 1 + y^2$. The Badlands are depicted in gray. Conjecturally, there are infinitely many non-cluster walls, lying at each rational slope within the gray cone. The wall-function on the red ray $\mathbb R_{\leq 0}(3, 2)$ is  $1+x^3 y^2 + 2x^6 y^4 + 5x^9 y^6 + 14 x^{12} y^8 + \cdots $. The coefficient 14 comes from Example~\ref{tight_exmp}(5).}
\label{fig: D32}
\end{figure}
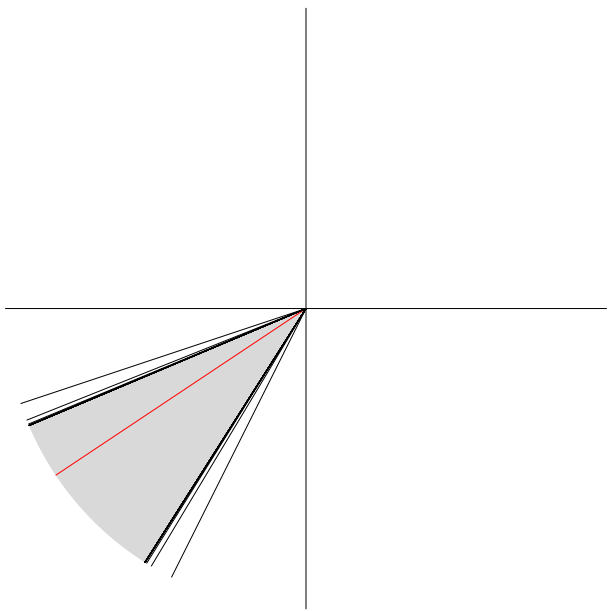

Our result from \Cref{sec: first main thm} allows one to understand directly the scattering diagrams $\mathfrak D$ with arbitrary power series as initial wall-functions
\[
    P_i(x) = 1 + p_{i, 1} x + \dots + p_{i, j} x^{j} + \dots \in \Bbbk\llbracket x \rrbracket, \quad i = 1, 2.
\]
The functions on the added rays $\mathbb R_{\leq 0}(a, b)$ with coprime $(a, b)\in \mathbb N^2$ are of the form
\begin{equation}\label{eq: wall function}
    f_{(a, b)} = 1 + \sum_{k\geq 1} \lambda(ka, kb)x^{ka}y^{kb}.
\end{equation}
\Cref{first_main_thm} then states that any $\lambda(ka, kb)$ is given by a weighted count of tight gradings. In particular, they are polynomials of $p_{i, j}$ with positive integer coefficients. Moreover using a \emph{change of lattice trick} and a \emph{perturbation trick} adapted from \cite{GPS,GHKK}, we are able to obtain the positivity in full generality for any consistent rank-2 scattering diagram obtained from \Cref{thm: consistent sd rank 2}.

\begin{thm}\label{thm: multiple initial lines}
    For an index set $S$ and an initial scattering diagram
    \[
        \mathfrak D_\mathrm{in} = \{(\mathbb Rm_i, P_i(x^{m_i}))\mid i\in S\}
    \]
    a (possibly infinite) collection of walls where $P_i = 1 + \sum_{j\geq 1} p_{i,j}x^j$, any coefficient of the wall-function of any added ray in $\mathfrak D\cut \mathfrak D_\mathrm{in}$ is in $\mathbb N[p_{i,j}\mid i\in S, j\geq 1]$.
\end{thm}

\subsection{Cluster scattering diagrams in higher ranks}\label{subsec: higher rank sd}

The rank-2 scattering structure can be extended to higher dimensions to form higher-rank scattering diagrams. They play crucial roles in mirror symmetry \cite{GS}, Donaldson--Thomas theory \cite{KSwcs}, and cluster algebras \cite{GHKK}. In this section, we focus on the scattering diagrams devised for generalized cluster algebras \cite{Mou, CKM}.

Let $B = (b_{ij})\in \operatorname{Mat}_{n\times n}(\mathbb Z)$ be skew-symmetrizable and $D = \operatorname{diag}(d_i)\in \operatorname{Mat}_{n\times n}(\mathbb Z_{>0})$ be a left skew-symmetrizer of $B$, that is, $(DB)^\intercal + DB = 0$. Let $N = \mathbb Z^n$ and $\{e_1, \dots, e_n\}$ denote the standard basis. Define a $\mathbb Q$-valued skew-symmetric bilinear form $\omega$ on $N$ by
\[
    \omega \colon N \times N \rightarrow \mathbb Q, \quad \omega(e_i, d_je_j) = b_{ij}.
\]
For technical simplicity, we assume $\omega$ to be \emph{non-degenerate}; if not, one can always find a super lattice $\widetilde N$ of higher rank and extend $\omega$ to a non-degenerate one on $\widetilde N$. Let $M = \operatorname{Hom}(N, \mathbb Z)$ be the dual lattice. Our scattering diagram to be defined will have walls living in $M_\mathbb R = M\otimes \mathbb R$.

Let $\ell_i$ be the greatest common divisor of the $i$-th column of $B$. The homomorphism
\[
    w_i \coloneqq \omega(-, \frac{d_i}{\ell_i}e_i) \colon N \rightarrow \mathbb Z
\]
is an element in $M$. Let $P$ be the monoid in $M$ generated by $\{w_i\mid i = 1, \dots n\}$. Since $\omega$ is non-degenerate, $P$ is contained in a strictly convex cone of $M_\mathbb R$. Analogous to the rank-2 case, we consider the monoid algebra $\Bbbk[P]$ and its completion $\widehat{\Bbbk[P]}$ at the maximal ideal $\mathfrak m$ generated by $\{x^p\mid p\in P\cut \{0\}\}$. The following definition generalizes the concept of a \emph{wall} to higher ranks.

\begin{defn}
    A \emph{wall} in $M_\mathbb R$ is a pair $(\mathfrak d, f_\mathfrak d)$ consisting of the \emph{support} $\frakd$ and the \emph{wall-function} $f_\frakd$ such that
    \begin{enumerate}
        \item $\mathfrak d$ is a codimension-one convex rational polyhedral cone contained in the hyperplane $n_0^\perp \subseteq M_\mathbb R$ for some primitive $n_0\in N^+ \coloneqq \{a_1e_1 + \dots + a_ne_n\mid a_i\in \mathbb N\}\cut \{0\}$;
        \item $f_\mathfrak d = f_\mathfrak d(x^w) = 1 + \sum_{k\geq 1} c_kx^{kw}$ for $c_k\in \Bbbk$ where $w$ is the primitive element in $P$ parallel to $\omega(-, n_0)$.
    \end{enumerate}
\end{defn}

In the current setup, \emph{wall-crossing automorphisms} are defined in the same way as in rank 2, which act on the complete algebra $\widehat{\Bbbk[P]}$. The definition of a scattering diagram and related notions discussed in rank 2 such as \emph{regular paths}, \emph{path-ordered products} and \emph{consistent scattering diagrams} generalize verbatim to higher ranks.

\begin{defn}
    A wall $(\mathfrak d, f_\mathfrak d)$ is called \emph{incoming} if $\mathfrak d = \mathfrak d + \mathbb R_{\geq 0}\omega(-, n_0)$ and otherwise \emph{non-incoming} or \emph{outgoing}.
\end{defn}

\begin{thm}[\cite{GS, KSwcs, GHKK}]\label{thm: consistent sd higher rank}
    For any initial scattering diagram $\mathfrak D_\mathrm{in}$ consisting of incoming walls, there is a unique consistent scattering diagram $\mathfrak D$ (up to equivalence) by adding only outgoing walls to $\mathfrak D_\mathrm{in}$.
\end{thm}

\begin{rem}
    Since there are many more possibilities in convex polyhedral cones in dimensions higher than two, we adopt \emph{equivalence relations} between scattering diagrams that allow, for example, subdivision of walls. Formally, two scattering diagrams are \emph{equivalent} if they define identical path-ordered products for an arbitrary path (that is regular to both).
\end{rem}

\begin{defn}[\cite{Mou, CKM}]\label{def: gen cluster sd}
    Let the initial scattering diagram $\mathfrak D_\mathrm{in}$ be the collection of incoming walls $\{(e_i^\perp, P_i(x^{w_i}))\mid i = 1, \dots, n\}$ with $P_i(x) = 1 + \sum_{k=1}^{\ell_i}p_{i,k}x^k$ for $p_{i, k}\in \Bbbk$. The \emph{generalized cluster scattering diagram} $\mathfrak D$ (associated to $B$, $D$, and the choices $p_{i, k}$) is defined to be the unique consistent scattering diagram (up to equivalence) for $\mathfrak D_\mathrm{in}$ guaranteed by \Cref{thm: consistent sd higher rank}.
\end{defn}

When each $P_i$ is the binomial $1 + x^{\ell_i}$, the scattering diagram $\mathfrak D$ is essentially the \emph{cluster scattering diagram} introduced by Gross--Hacking--Keel--Kontsevich \cite{GHKK} as a fundamental tool in their systematic study of canonical (or theta) bases of cluster algebras. They proved that $\mathfrak D$ possesses a positivity that leads to the Laurent positivity of all cluster variables and to the strong positivity of theta bases which contain all cluster monomials.

Built on our tight grading formula (\Cref{first_main_thm}) in rank 2, we obtain a positivity of generalized cluster scattering diagrams in higher ranks (\Cref{thm: positivity higher rank}). This subsequently confirms the conjectural Laurent positivity (\Cref{thm: positivity cluster variables}) of cluster variables for Chekhov--Shapiro's generalized cluster algebras \cite{CS} and manifests the strong positivity of theta bases (\Cref{thm: strong positivity}). More precise statements are to come in the next sections.

\section{The second main theorem}\label{sec: second main theorem}

The generalized cluster scattering diagram $\mathfrak D$ from \Cref{def: gen cluster sd} admits the following positivity on its wall-functions.

\begin{thm}\label{thm: positivity higher rank}
    Let $\mathfrak D$ be a generalized cluster scattering diagram of any rank.  There is a representative of $\mathfrak D$ such that any wall-function $f_\frakd = 1 + \sum_{k\geq 1} c_k x^{kw}$ has positive coefficients in the sense that any $c_k$ is a polynomial of the initial coefficients with positive integer coefficients, that is, it belongs to $\mathbb N[p_{i, j}\mid 1\leq i \leq n, 1 \leq j \leq \ell_i]$.
\end{thm}

\begin{proof}[Proof sketch]
    We will construct a sequence of finite scattering diagrams $\mathfrak D_k$ with respect to $\Bbbk[P]/\mathfrak m^{k+1}$ so that the limit of $(\mathfrak D_k)_{k\geq 1}$ is equivalent to $\mathfrak D$. 
    The theorem is proven by showing inductively that each $\mathfrak D_k$ admits positivity.

    Let $\mathfrak D_1 = \mathfrak D_\mathrm{in}^{\leq 1}$ where $(\cdot)^{\leq k}$ means to truncate all wall-functions up to order $k$, that is, modulo $\mathfrak m^{k+1}$. Suppose that $\mathfrak D_k$ has been constructed to satisfy that
    \begin{enumerate}
        \item $\mathfrak D_k = \mathfrak D_k^{\leq k}$ and it is consistent modulo $\mathfrak m^{k+1}$;
        \item any wall-function has positive coefficients in the same sense as in the theorem;
        \item any non-initial wall $\frakd$ is outgoing and only one facet can be a perpendicular joint (\cite[Def. 2.11]{AG}, \cite[Def.-Lem. C2]{GHKK});
        \item any perpendicular joint $\mathfrak j$ is contained in the facet of at most one outgoing wall in each possible direction.
    \end{enumerate}

    Consider at each perpendicular joint $\mathfrak j$ of $\frakD_k$ the scattering diagram $\mathfrak D_\mathfrak j$ of walls containing $\mathfrak j$. It is essentially a rank-2 scattering diagram when projected to the quotient of $M_\mathbb R$ by $T_\mathfrak j$ the tangent space of $\mathfrak j$. By (3) and (4), the projection consists of lines and only distinct outgoing rays. It is consistent modulo $\mathfrak m^{k+1}$ by (1) but not necessarily so modulo $\mathfrak m^{k+2}$. A crucial step is using \Cref{thm: multiple initial lines} to complete $\mathfrak D_\mathfrak j$ to achieve the consistency around $\mathfrak j$ modulo $\mathfrak m^{k+2}$. We add new outgoing walls with support $\frakd = \mathfrak j - \mathbb R\omega(-, n_0)$ and wall-function $f_\frakd$ only non-trivial in order $k+1$ for distinct $n_0$ in $T_\mathfrak j^\perp \cap N^+$. Denote this collection of walls by $\frakD(\mathfrak j)$. For an existing outgoing wall $\frakd$ in $\frakD_\mathfrak j$, we replace $f_\frakd$ with $f_\frakd'$ by only adding positive terms in $\mathfrak m^{k+2}\cut\mathfrak m^{k+3}$. Namely, we define
    \[
        \mathfrak D_{k+1} \coloneqq \mathfrak D_\mathrm{in}^{\leq k+1} \cup \bigcup_{\mathfrak j} \left(\mathfrak D(\mathfrak j) \cup \{(\frakd, f_\frakd')\mid \mathfrak j\subset \partial \frakd\} \right)
    \]
    where the second union is over all perpendicular joints in $\mathfrak D_k$.

    Every condition from (1) to (4) is not hard to check for $\mathfrak D_{k+1}$ except the consistency modulo $\mathfrak m^{k+2}$. The strategy is to compare $\mathfrak D_{k+1}$ with $\mathfrak D^{\leq k+1}$ which is consistent and only has finitely many outgoing walls. They are equivalent modulo $\mathfrak m^{k+1}$. Their difference $\frakD'$, easily realized as a scattering diagram only non-trivial in order $k+1$, has possibly only outgoing walls with only parallel joints contained in their facets because $\mathfrak D_{k+1}$ is already consistent modulo $\mathfrak m^{k+2}$ at any perpendicular joint. Should there be any non-trivial $(\frakd, f_\frakd)\in \frakD'$ with normal vector $n_0$, the whole line $\mathbb R\omega(-, n_0)$ lies in the minimal face, forcing $\frakd$ to be incoming, a contradiction. This also shows the equivalence between $\mathfrak D_{k+1}$ and $\mathfrak D^{\leq k+1}$.

    Finally, the limit of $(\mathfrak D_k)_{k\geq 1}$ is clearly equivalent to $\mathfrak D$ and inherits the positivity from that of each $\mathfrak D_k$.
\end{proof}

The application of \Cref{thm: positivity higher rank} in the generalized cluster algebras of Chekhov and Shapiro \cite{CS} will be discussed in the next \Cref{sec: broken lines and theta functions}.

\section{Broken lines, theta functions and positivity}\label{sec: broken lines and theta functions}

Broken lines were introduced by Gross \cite{G10} in his study of mirror symmetry of $\mathbb P^2$. They have been used to construct theta functions on the mirror object in more general situations as developed for instance in \cite{GHK} and \cite{CPS}. These concepts are crucial to our study of rank-2 scattering diagrams and higher-rank generalized cluster algebras. We first review their definitions for a rank-2 scattering diagram $\mathfrak D$ introduced in \Cref{subsec: rank 2 sd}. Suppose that the walls in $\frakD$ have disjoint supports except at the origin.

\begin{defn}
      Let $m_0 \in M$ and a general point $Q \in M_\R \cut \Supp(\frakD)$. A \emph{broken line} $\gamma$ for $m_0$ with endpoint $Q$ is a piecewise linear continuous proper path $\gamma\colon (-\infty,0] \to M_\R \cut \{0\}$ with values $-\infty = \tau_0 < \tau_1 < \cdots < \tau_\ell = 0$ and an associated monomial $c_i x^{m_i} \in \Bbbk[M]$ for each $i = 0, \dots, \ell$ such that
    \begin{enumerate}
        \item $\gamma(0) = Q$ and $c_0 = 1$;
        \item $\dot{\gamma}(\tau) = -m_i$ for any $\tau \in (\tau_{i-1},\tau_i)$ for each $i = 1, \dots, \ell$;
        \item the subpath $\gamma\vert_{(\tau_{i-1}, \tau_{i+1})}$ transversally crosses (the support of) a wall $(\frakd_i, f_i)$ at $\tau_i$ for $i = 1, \dots, \ell-1$;
        \item for $i = 0, \dots, \ell -1$, $c_{i+1}x^{m_{i+1}}$ is a monomial term of
        \[
            \mathfrak p_{-m_i, \frakd_i}(c_i x^{m_i}) = c_ix^{m_i} \cdot f^{n_i\cdot m_i}_{i}
        \]
        where $n_i$ is primitive in $N$ and orthogonal to $\frakd_i$ in the direction $n_i \cdot m_i>0$.
    \end{enumerate}
\end{defn}

We call the final coefficient $c_\ell$ the \emph{weight} of a broken line $\gamma$ and denote $c(\gamma) \coloneqq c_\ell$. We refer to each $\gamma(\tau_i)$ for $i = 1,\dots,\ell - 1$ as a \emph{bending of multiplicity $k$} of $\gamma$ at the wall $\frakd_i$, where $c_{i+1}x^{m_{i+1}}$ is the $(k+1)$-th term of $c_ix^{m_i} f_i^{n_i \cdot m_i}$ (ordered increasingly by exponent). Each element $m_i \in M$ is referred to as the \emph{exponent} of $\gamma$ on the corresponding linear domain, and we let $m(\gamma)$ denote the final exponent $m_\ell \in M$.

\begin{defn}\label{def: theta function rank 2}
   The \emph{theta function} associated to $m_0$ and $Q$ is
   \[
        \vartheta_{Q, m_0} = \sum_\gamma c(\gamma) x^{m(\gamma)}
   \]
   where the sum is over all broken lines for $m_0$ with endpoint $Q$.
\end{defn}

\begin{rem}
    The point $Q$ is always chosen general enough for broken lines to avoid the origin (the only singular locus in rank 2). The sum is generally infinite but indeed lies in $x^{m_0}\widehat{\Bbbk[P]}$ \cite[Proposition 3.4]{GHKK}.
\end{rem}

\begin{exmp}\label{exmp: bl weight}
    The broken line $\gamma$ for $m_0 = (-12, -11)$ in \Cref{fig: broken line q weight} (with $\ell = 4$) has exponents
    \[
        m_1 = (-6, -7),\ m_2 = (-2, -5),\ m_3 = (-2, -1) 
    \]
    and coefficients
    \[
        c_1 = 1,\ c_2 = 9,\ c_3 = 72,\ c(\gamma) = c_4 = 72.
    \]
\end{exmp}

Now we turn to the situation where $\mathfrak D$ is determined by two polynomials $P_1$ and $P_2$ as in (\ref{eq: initial sd rank 2}). In this case the first quadrant of $\mathbb R^2$ is a connected component of $M_\mathbb R\cut \Supp(\frakD)$. It follows from a general theorem of \cite{CPS} that $\vartheta_{m_0} = \vartheta_{Q, m_0}$ for general $Q$ in the first quadrant is (a finite sum in the current case) independent of $Q$.

\begin{thm}[Strong positivity in rank 2]\label{thm: strong positivity rank 2}
    The set $\{\vartheta_{m}\mid m\in M\}$ is a basis of their linear span $\mathcal A(P_1, P_2)$ in $\Bbbk[M]$ which is closed under multiplication, hence making $\mathcal A(P_1, P_2)$ an algebra. The multiplication constants are defined in the same way as \cite[Def.-Lem. 6.2]{GHKK} and are manifestly polynomials in $p_{i,j}$ with positive integer coefficients due to \Cref{first_main_thm}.
\end{thm}

\begin{rem}
    For the case of $\frakD_{(\ell_1,\ell_2)}$, the above theorem is the rank-2 special case of \cite{GHKK} whose method extends to our case except for the positivity. The algebra $\mathcal A(P_1, P_2)$ in fact secretly equals a generalized cluster algebra, and the so-called theta basis $\{\vartheta_m\mid m\in M\}$ is identical to the combinatorially defined \emph{greedy basis} containing all cluster monomials; see \Cref{sec: greedy basis} and \Cref{sec: proof first main theorem}.
\end{rem}

Turning to higher ranks, we now give a minimalistic definition of generalized cluster algebras of Chekhov--Shapiro \cite{CS} and then discuss the application of broken lines and theta functions. Recall the notations in \Cref{def: gen cluster sd}. For technical simplicity, we assume that each $P_i$ is monic. Let $((x_i)_{i=1}^n, (P_i)_{i=1}^n, B)$ be an \emph{initial seed} associated to a root $t_0$ of the $n$-regular (infinite) tree $\mathbb T_n$ where the $n$ edges incident to any vertex are distinctly labeled by $\{1, \dots, n\}$. This determines an assignment of a \emph{seed} $((x_{i;t})_i, (P_{i;t})_i, (b_{ij}^t)_{i,j})$ to each vertex $t\in \mathbb T_n$ by imposing the \emph{mutation} rule that for any $t \frac{k}{\quad\quad} t'$,
\[
    x_{k;t'} = x_{k;t}^{-1} \prod_{j=1}^n x_{j;t}^{[-b_{jk}^t]_+/\ell_i} P_{k;t}\left(\prod_{j=1}^n x_{j;t}^{b_{jk}^t} \right)
\]
and $x_{j;t'} = x_{j;t}$ for $j\neq k$; $P_{k;t'}(x) = x^{r_k} P_{k;t}(1/x)$ and $P_{j;t'} = P_{j;t}$ for $j\neq k$; the matrices $(b_{ij}^{t'})$ and $(b_{ij}^t)$ are the Fomin--Zelevinsky mutations of each other in direction $k$. The tuple $(x_{i;t})_i$ is called a \emph{cluster} and each $x_{i; t}$ a \emph{cluster variable}, which by the recursive definition is a rational function in $(x_i)_i$.

The \emph{generalized cluster algebra} $\mathcal A$ is defined to be the algebra generated by the rational functions $x_{i;t}$. Fomin--Zelevinsky's cluster algebra \cite{FZ} is recovered when each $P_i$ is a binomial. Identify $x_i$ with $x^{e_i^*}$ where $\{e_i^*\mid i = 1, \dots, n\}$ denotes the dual standard basis in $M$. Every $x_{i; t}$ now becomes an element in $\operatorname{Frac}(\Bbbk[M])$.

Broken lines and theta functions can be defined in higher ranks. As the cluster case is the most relevant to us, we refer the reader to \cite{GHKK} for full details and to \cite{Mou} and \cite{CKM} for the generalized case. Roughly, for each point $m_0\in M$ and a general point $Q$ in the positive orthant of $M_\mathbb R$, one can consider the set of all broken lines in $\mathfrak D$ with initial exponent $m_0$ ending at $Q$ and define the associated theta function $\vartheta_{m_0}$ to be the formal sum as in \Cref{def: theta function rank 2}, which is independent of $Q$. It is possible that $\vartheta_{m_0}$ is an infinite sum but it always lies in $x^{m_0}\widehat{\Bbbk[P]}$. However, we have

\begin{thm}[\cite{GHKK,Mou,CKM}]\label{prop: cluster variable theta function}
    Every cluster variable $x_{i;t}$ is a theta function which is a finite sum of terminal monomials of broken lines.
\end{thm}

\begin{figure}
\begin{tikzpicture}[scale=9]
\def\v{0.15cm}
\draw[draw=none, fill=gray] (0,0) --  ({atan(1.04)}:{-.4}) arc({atan(1.04)}:{atan(0.965)}:{-.4}) -- cycle;
\foreach \x in {2,...,3}
    {
    \draw[line width=1pt] (0,0)--({atan((\x-1)/\x)}:{-.4 });
    }
\foreach \x in {4,...,30}
    {
    \draw[line width=0.2pt,color=gray,draw opacity=0.4] (0,0)--({atan((\x-1)/\x)}:{-.4 });
    }
\foreach \x in {2,...,30}
    {
    \draw[line width=0.1pt,color=gray,draw opacity=0.4] (0,0)--({atan(\x/(\x-1))}:{-.4 });
    }
\draw[line width=1pt] (-.4,0)--(.4,0);
\draw[line width=1pt] (0,-.4)--(0,.4);
\draw[line width=1pt, color=blue] (0.25,0.375) node[circle,fill=black,minimum size=\v,inner sep=0pt]{}--(0,0.25) node[circle,fill=blue,minimum size=\v,inner sep=0pt]{}--(-1/8,-1/16) node[circle,fill=blue,minimum size=\v,inner sep=0pt]{}--(-1/6,-1/9) node[circle,fill=blue,minimum size=\v,inner sep=0pt]{}--(-.3147,-0.2468);

\draw (-.54,-.25) node[anchor=west,color=blue]  {\small$(-12,-11)$};
\draw (-.37,-.07) node[anchor=west,color=blue]  {\small$(-6,-7)$};
\draw (-.25,0.12) node[anchor=west,color=blue]  {\small$(-2,-5)$};
\draw (0,0.37) node[anchor=west,color=blue]  {\small$(-2,-1)$};

\end{tikzpicture}
\caption{A broken line $\gamma$ with endpoint $(2,\,3)$ is depicted on the scattering diagram $\frakD_{(2,2)}$ (shown in \Cref{fig: D2}).   The exponent of each domain of linearity of $\gamma$ is shown as an element of $\Z^2$.  The weight of $\gamma$ is computed in \Cref{exmp: bl weight}. }
\label{fig: broken line q weight}
\end{figure}
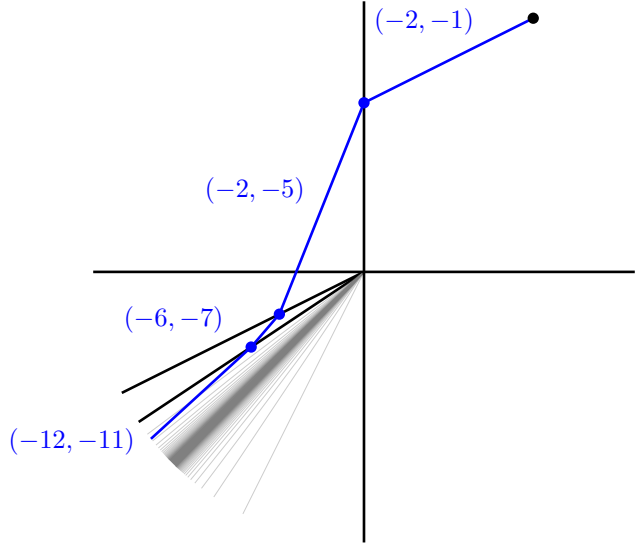

\begin{figure}
\begin{tikzpicture}[scale=.4]
\def\v{0.15cm}
\draw[draw=none, fill=gray, opacity = 0.3] (0,0) --  ({atan(1.57)}:{-10}) arc({atan(1.57)}:{atan(0.424)}:{-10}) -- cycle;
\draw[line width=0.1pt] (-10,0)--(10,0);
\draw[line width=0.1pt] (0,-10)--(0,10);
\foreach \x in {1,...,12}   {
    \draw[line width=0.2pt, color=gray,draw opacity=0.4] (0,0)--({atan(\slope[\x])}:{-10});
    }
\draw[line width=1pt] (0,0)--({atan(0.666666666)}:{-10});

\draw[line width=1pt, color=blue] (7,3) node[circle,fill=black,minimum size=\v,inner sep=0pt]{}--(-3,-2) node[circle,fill=blue,minimum size=\v,inner sep=0pt]{}--(-8.48543,-5.29126);

\draw (-11.5,-4.5) node[anchor=west,color=blue]  {\small$(-8,-5)$};
\draw (1,2.5) node[anchor=west,color=blue]  {\small$(-2,-1)$};

\end{tikzpicture}

\caption{Suppose $P_1 = 1 + x^3$ and $P_2 = 1 + y^2$, as in \Cref{fig: D32}.  A broken line is depicted on $\frakD(P_1,P_2)$ in blue.  This broken line has initial exponent $(-m,-n) = (-8,-5)$, bends with multiplicity 2 at the wall of slope $(a,b) = (3,2)$, and terminates at $(7,3)$.  This is the unique broken line with this fixed initial exponent, endpoint, and final exponent that does not bend at the $x$-axis.  The weight of this broken line is $2$, corresponding to the fact that there are two tight gradings $\omega: E(M_{-1}(6,4)) \to \Z_{\geq 0}$, each with weight $1$.}
\label{fig: D32 broken line}
\end{figure}
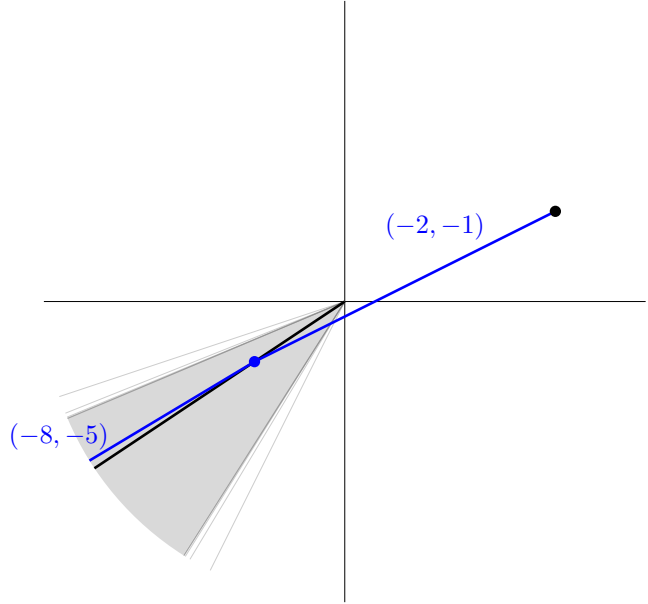

Combining \Cref{prop: cluster variable theta function} with \Cref{thm: positivity higher rank} the positivity of scattering diagrams, we derive immediately

\begin{thm}[Positive Laurent phenomenon]\label{thm: positivity cluster variables}
    Every cluster variable $x_{i;t'}$ is a Laurent polynomial of the cluster variables in any other cluster $(x_{i;t})_i$ with coefficients being polynomials of $(p_{i,j}^t)_{i, j}$ (the coefficients of $P_{i;t}(x)$) with positive integer coefficients.
\end{thm}

The above theorem gives an affirmative answer to \cite[Conjecture 5.1]{CS} (which imposes more restrictive assumptions). The positivity theorem of Gross--Hacking--Keel--Kontsevich \cite[Corollary 0.4]{GHKK} follows as a special case when each $P_i$ is a binomial.

The strong positivity in higher ranks is more subtle but again it comes from the positivity \Cref{thm: positivity higher rank} of scattering diagrams. While a more complete form can be presented similar to \cite[Theorem 0.3]{GHKK}, here we describe a snapshot. Let $\Theta\subseteq M$ denote the subset such that $\vartheta_{m}$ for $m\in \Theta$ is a finite sum coming from broken lines.

\begin{thm}[Strong positivity of theta bases]\label{thm: strong positivity}
    The set $\{\vartheta_{m}\mid m\in \Theta\}$ form a basis of their $\Bbbk$-span in $\Bbbk[M]$ which is closed under multiplication. The multiplication constants as defined in \cite[Def.-Lem. 6.2]{GHKK} are polynomials in $p_{i,j}$ with positive integer coefficients due to \Cref{thm: positivity higher rank}.
\end{thm}

Finally, we note that one can always factor $P_i$ into binomials in an algebraic closure $\overline \Bbbk$ of $\Bbbk$. Then a weaker positivity of cluster variables was proven in \cite{Mou} in terms of the roots of $P_i$ in the same fashion as \cite{GHKK}; see the discussion \cite[Section 8.5]{Mou}. However, there are generalized cluster structures where the coefficients $p_{i,k}$ naturally arise as regular functions on algebraic varieties \cite{GSV18}, which motivates the current form of Laurent positivity presented in \Cref{thm: positivity cluster variables}.

\section{Generalized greedy bases}\label{sec: greedy basis}

In this section, we introduce a class of rank-2 generalized cluster algebras with more general coefficient system than \Cref{sec: broken lines and theta functions}. This is needed for the purpose of studying scattering diagrams.

Let $\Bbbk$ be any field of characteristic zero. We consider polynomials in $\Bbbk[p_{1,1},...,p_{1,\ell_1},p_{1,\ell_1}^{-1}][z]$ and $\Bbbk[p_{2,1},...,p_{2,\ell_2},p_{2,\ell_2}^{-1}][z]$. Let
\begin{align*}
    & P_1 = 1 + p_{1,1}z + \cdots + p_{1,\ell_1-1} z^{\ell_1-1} + p_{1,\ell_1}z^{\ell_1} \quad \text{and} \\
    & P_2 = 1 + p_{2,1}z + \cdots + p_{2,\ell_2-1} z^{\ell_2-1} + p_{2,\ell_2}z^{\ell_2}.
\end{align*}
We define further $\overline{P}_1(z):=\frac{1}{p_{1,\ell_1}} z^{\ell_1}P_1(z^{-1})$ and $\overline{P}_2(z):=\frac{1}{p_{2,\ell_2}} z^{\ell_2}P_2(z^{-1})$.

Consider the ring $\Bbbk(p_{1,1},...,p_{1,\ell_1},p_{2,1},...,p_{2,\ell_2})(x_1,x_2)$ of rational functions. We inductively define rational functions $x_k\in \Bbbk(p_{1,1},...,p_{1,\ell_1},p_{2,1},...,p_{2,\ell_2})(x_1,x_2)$ for $k\in\mathbb{Z}$ by the rule:
$$
x_{k+1}x_{k-1}=\left\{\begin{array}{ll}
P_1(x_k) &\text{if }k\equiv 1 \text{ mod }4;\\
P_2(x_k) &\text{if }k\equiv 2 \text{ mod }4;\\
\overline{P}_1(x_k) &\text{if }k\equiv 3 \text{ mod }4;\\
\overline{P}_2(x_k) &\text{if }k\equiv 0 \text{ mod }4.\\
\end{array}\right.
$$
Each $x_k$ is called a \emph{cluster pre-variable}.
\begin{lem}
   Each $x_k$ is  a  Laurent polynomial in 
   $$\underline{\Bbbk}:=\Bbbk[p_{1,1},...,p_{1,\ell_1},p_{1,\ell_1}^{-1},p_{2,1},...,p_{2,\ell_2},p_{2,\ell_2}^{-1}][x_1^{\pm1},x_2^{\pm1}]\,.$$
\end{lem}

 Let $X_k$ be the element such that the coefficient of the (unique) lowest (total) degree term of $X_k$ is equal to 1, and $X_k = p_{1,\ell_1}^{a_k}p_{2,\ell_2}^{b_k}x_k$ for some $a_k,b_k\in\mathbb{Z}$.  These $X_k$ are called the \emph{generalized cluster variables}. Let the \emph{generalized cluster algebra} $\mathcal{A}(P_1,P_2)$ be the $\Bbbk[p_{1,1},...,p_{1,\ell_1},p_{2,1},...,p_{2,\ell_2}]$-subalgebra of $\underline{\Bbbk}$ generated by the set $\{X_k\}_{k\in \mathbb{Z}}$. As the name suggests, it actually equals the $\Bbbk$-algebra $\mathcal A(P_1, P_2)$ in \Cref{thm: strong positivity rank 2} when $p_{i,j}$'s are evaluated in $\Bbbk$ (\Cref{subsec: polynomial case}).

Lee, Li and Zelevinsky defined \emph{greedy bases} \cite{LLZ} for ordinary rank-2 cluster algebras. Rupel \cite{Rupgengreed} constructed greedy bases for generalized rank-2 cluster algebras when $P_1$ and $P_2$ are monic and palindromic. We are ready to give a definition in our more general case. We use here the notation $[a]_+ = \max(a,0)$.

\begin{defn}
For each $(a_1, a_2) \in \mathbb{Z}^2$, we define the associated \emph{greedy element} $x[a_1,a_2]$ by
$$
    x[a_1,a_2] \coloneqq x_1^{-a_1} x_2^{-a_2} \sum_{\omega} \text{wt}(\omega) x_1^{\omega(E_2)}x_2^{\omega(E_1)},
$$
where the sum is over all compatible (not necessarily tight) gradings 
\[
    \omega: E = E([a_1]_+, [a_2]_+)\longrightarrow\mathbb{Z}_{\ge0}.
\]
\end{defn}

\begin{thm}
\label{mtheorem}

The greedy elements $x[a_1,a_2]$ for $(a_1, a_2) \in \mathbb{Z}^2$ form a $\Bbbk[p_{1,1},...,p_{1,\ell_1},p_{2,1},...,p_{2,\ell_2}]$-basis for $\mathcal{A}(P_1,P_2)$, which we refer to as the greedy basis.
\end{thm}
\begin{proof}[Proof Sketch]
The work of Rupel \cite{Rupgengreed} can be adapted to our notion of generalized cluster algebras, where we allow for non-monic, non-palindromic polynomials.  The adaptation includes working the cluster pre-variables rather than the cluster variables, sometimes exchanging $P_i$ with $\overline{P_i}$, and adding factors of $p_{i,\ell_i}$ for $i \in \{1,2\}$.  Then we must consider several variants of the greedy elements, depending on whether the edges are weighted by coefficients of $P_i$ or $\overline{P_i}$.  The actions $\sigma_1$ and $\sigma_2$ map between these variants, and from this we can obtain that our elements satisfy the greedy recursion \cite[Definition 2.11]{Rupgengreed}.
\end{proof}

\section{Proof sketch of the first main theorem}\label{sec: proof first main theorem}
We now outline the proof of our first main theorem, \Cref{first_main_thm}, in the scattering diagram context of \Cref{sec: sd}. Namely, we will show that the coefficient $\lambda(ka,kb)$ in (\ref{eq: wall function}) is given by the tight grading formula in \Cref{first_main_thm}. The formula is derived from matching the greedy element $x[m, n]$ from \Cref{sec: greedy basis} with the theta function $\vartheta_{(-m, -n)}$ from \Cref{sec: broken lines and theta functions}. The proof is given first for the ordinary cluster case, i.e., when $P_1 = 1 + x^{\ell_1}$ and $P_2 = 1 + y^{\ell_2}$, where $x[m, n] = \vartheta_{(-m, -n)}$ is the main theorem of \cite{CGM}. Then we explain the necessary ingredients for the generalized case where $P_1$ and $P_2$ are polynomials. Finally, the power series case is handled by taking the limit.

\subsection{The cluster algebra case} We show that, for a suitable choice of initial and final exponent, there is exactly one broken line that does not bend at the $x$-axis.  Moreover, for any $k, a, b > 0$ with $\gcd(a,b) = 1$, we can choose the initial and final exponent such that this broken line bends only at the wall of slope $\frac{b}{a}$ with resulting weight $\lambda(ka,kb)$.  We then use the equality of the greedy and theta bases \cite{CGM} to show that the size of the appropriate set of tight gradings gives the weight of this broken line.

Let $P_1=1+x^{\ell_1}$ and $P_2=1+y^{\ell_2}$. By symmetry, it is enough to handle the case of \emph{positive angular momentum}.  That is, we can fix an endpoint $Z = (Z_1,Z_2) \in \R^2_{>0}$ such that the quantity $(-m+ka)Z_2 -(-n+kb)Z_1$ is positive.  Let $\BL^t_+(m,n,ka,kb)$ denote the set of broken lines with initial exponent $(-m,-n)$ and final exponent $(-m+ka, -n+kb)$ that bend at the $x$-axis with multiplicity $t$ and terminate at $Z$.  Given a set $S$ (of compatible gradings or broken lines), let $|S|$ denote the sum of the weights of the objects in $S$.  Let $\frakd$ denote the wall of slope $b/a$.  

\begin{rem}
Throughout this section, we require that for $(m,n) = M_{-1}(ka,kb)$, we have $m > ka$ and $n > kb$, which we denote by $(m,n) > (ka,kb)$.  This condition is necessary to ensure the existence of the corresponding broken line terminating in the first quadrant.  While \Cref{def: tight grading} allows $m = ka$ or $n = kb$, strict inequality can be obtained by replacing $(m,n)$ with $(m+kb, n + ka)$ without affecting the weighted sum of compatible gradings.
\end{rem}

\begin{lem}\label{lem: unique broken line}
    If $(m,n) = M_{-1}(ka,kb) > (ka,kb)$, then $\BL^0_+(m,n,ka,kb)$ consists of a single broken line.  This broken line bends with multiplicity $k$ at the wall of slope $\frac{b}{a}$ with no other bendings.
\end{lem}
\begin{proof}
By assumption, we have $(-b,a) \cdot (-m,-n) = 1$,  so $(-kb, ka)\cdot(-m,-n) = k$ for all $k \in \N$.  Since we must have $\gcd(m,n) = 1$, any $w_1,w_2> 0$ satisfying $(-w_2,w_1) \cdot (-n,-m) = k$ must be of the form $(w_1,w_2) = (ka,kb) + \alpha(m,n)$ for some integer $\alpha$.  If $\alpha > 0$, we have $w_1 > m$ so a broken line of exponent $(-m,-n)$ cannot bend at the wall $(w_1,w_2)$.  If $\alpha < 0$, then we have $\frac{w_2}{w_1} > \frac{b}{a}$ so $\gamma$ would cross the wall $\frac{w_2}{w_1}$ only after it crosses the wall of slope $\frac{b}{a}$. 

Thus $\gamma$ can bend only at walls of slope at least $\frac{b}{a}$. Unless it bends only at the wall of slope $\frac{b}{a}$, then it must bend at the $x$-axis to have the correct final exponent.
\end{proof}

\begin{lem}
If $(m,n) = M_{-1}(ka,kb) > (ka,kb)$, then $|\BL^0_+(m,n,ka,kb)| = \lambda(ka,kb)$. 
\end{lem}
\begin{proof}
 The wall-function for $\frakd$ can be expressed as 
$$f_\frakd= 1 + \sum_{k \geq 1} \lambda(ka,kb) x^{kb}y^{ka}\,.$$

Let $\gamma$ be the unique broken line in $\BL^0_+(m,n,ka,kb)$, as described in \Cref{lem: unique broken line}.  By hypothesis, we have $(m,n) \cdot (-b,a) = 1$.  Thus, the weight associated to the bending of $\gamma$ over $\frakd$ with multiplicity $k$ is given by the coefficient of $x^{kb}y^{ka}$ in $f_{\frakd}$, which is precisely $\lambda(ka,kb)$.  Since $\gamma$ is the unique broken line in $\BL^0_+(m,n,ka,kb)$ and $\gamma$ only bends once, we can conclude that $|\BL^0_+(m,n,ka,kb)| = |\gamma| = \lambda(ka,kb)$, as desired.
\end{proof}

Let $\CG^t_+(m,n,ka,kb)$ denote the set of compatible gradings on $\calP(m,n)$ such that the sum of the grading on the vertical edges is $ka$, the sum of the grading on the horizontal edges is $kb$, the sum of the grading on the horizontal edges outside $\sh(S_2)$ is $t$. 

\begin{rem} In the cluster algebra case, a tight grading has weight $1$ if it takes values in $\{0,
\ell_1\}$ on the vertical edges and $\{0,\ell_2\}$ on the horizontal edges, and weight $0$ otherwise.  If we restrict to considering only these gradings, then the total weight of gradings in $\CG^t_+(m,n,ka,kb)$ is the same as its cardinality.
\end{rem}

\begin{prop}
    If $(m,n) = M_{-1}(ka,kb) > (ka,kb)$ and $k,t \geq 0$, we have 
    $$|\BL^t_+(m,n,ka,kb)| = |\CG_+^t(m,n,ka,kb)|\,.$$
\end{prop}

\begin{proof}
We proceed by induction on the quantity $kb-t$.  The base case $kb - t = 0$ is clear, as both quantities are simply a product of two binomial coefficients.  It is readily seen that 
$$|\CG(m,n,ka,kb-t)| = \sum_{s=0}^t |\CG^s_+(m,n,ka,kb-t)|$$ and $$|\BL(m,n,ka,kb-t)| = \sum_{s=0}^t |\BL^s_+(m,n,ka,kb-t)|\,.$$
Since $|\CG(m,n,ka,kb-t)| = |\BL(m,n,ka,kb-t)|$ by \cite{CGM}, we have $|\CG^t_+(m,n,ka,kb-t)| = |\BL^t_+(m,n,ka,kb-t)|$.  We can then use
$$|\CG^t_+(m,n,ka,kb)| = \binom{[m-t\ell_1]_+}{t}|\CG^t_+(m,n,ka,kb-t)|$$
and 
$$|\BL^t_+(m,n,ka,kb)| = \binom{[m-t\ell_1]_+}{t}|\BL^s_+(m,n,ka,kb-t)|$$
to reach the desired equality.
\end{proof}

Since $\CG^0_+(m,n,ka,kb)$ is precisely the set of tight gradings, \Cref{first_main_thm} follows directly.

\subsection{The polynomial case}\label{subsec: polynomial case}

In order to prove \Cref{first_main_thm} when $P_1$ and $P_2$ are polynomials with constant term $1$, we use Rupel's construction of the generalized greedy basis and mimic the approach of \cite{CGM}.

Rupel characterized the smallest lattice quadrilateral $R_{a_1,a_2}$ containing the support of each generalized greedy element $x[a_1,a_2]$ \cite[Proposition 4.22]{Rupgengreed}.  Rupel's proof carries through almost identically, though one case relies on the interpretation of the greedy coefficients as sums of weights of compatible gradings, following from \Cref{mtheorem}.  We then consider a generalization of a result of Cheung, Gross, Muller, Musiker, Rupel, Stella, and Williams \cite{CGM} characterizing the greedy elements in terms of their support.  Let $R^\circ_{a_1,a_2}$ denote the smallest lattice quadrilateral containing the origin $(0,0)$ and $R_{a_1,a_2}$, where the interiors of the two edges incident to $(0,0)$ are excluded when the quadrilateral is not degenerate (see \cite[Figure 1]{CGM}).

\begin{lem}{\cite[Scholium 2.6]{CGM}}\label{lem: greedy support containment}
If $z \in \calA(P_1,P_2)$ is any element containing the monomial $x_1^{-a_1}x_2^{-a_2}$ with coefficient $1$ and whose support is contained in $R^\circ_{a_1,a_2}$, then $z = x[a_1,a_2]$.
\end{lem}

Thus, it is enough to show that each generalized theta basis element $\vartheta_{(-a_1, -a_2)}$ has support within $R^\circ_{a_1,a_2}$.  This is shown in the classical setting in \cite[Section 5]{CGM}, and their methods work essentially verbatim in the generalized case.  Their results do not rely on the structure of the scattering diagram, other than that all walls are along the coordinate axes or in the third quadrant (in the $d$-vector version), which holds in the generalized case.  This culminates in bounds on the final exponent of the broken lines contributing to the generalized theta basis, yielding the desired bounds on the support of the theta basis elements.

\subsection{The power series case}
For $P_{1} = 1 + \sum_{i\ge 1} p_{1,i}x^i$ and $P_{2} = 1 + \sum_{i\ge1} p_{2,i}y^i$, let $\mathcal{D}_{P_1,P_2}$ denote the resulting consistent scattering diagram. 
For each $t\in\mathbb{Z}_{>0}$, let $P_{1,t} = 1 + \sum_{i=1}^t p_{1,i}x^i$ and $P_{2,t} = 1 + \sum_{i=1}^t p_{2,i}y^i$. Then let $\mathcal{D}_{P_{1,t},P_{2,t}}$ denote the resulting consistent scattering diagram. Let $\omega$ be a tight grading such that its weight with respect to $(P_{1,t},P_{2,t})$ is zero, but its weight with respect to $(P_{1,t+1},P_{2,t+1})$ is nonzero. Because of the homogeneity constraint, the corresponding monomial is in $\mathfrak{m}^{t+1}$. Hence
$$\aligned
\mathcal{D}_{P_1,P_2} &= \lim_{t\rightarrow \infty} \mathcal{D}_{P_{1,t},P_{2,t}}\, \text{ mod }\mathfrak{m}^{t+1}\\
&=\lim_{t\rightarrow \infty} \mathcal{D}_{P_{1,t},P_{2,t}}. 
\endaligned
$$

\section{Future directions}

\subsection{Combinatorial problems}
We expect that tight gradings have rich combinatorics. One may pose a number of purely combinatorial questions. For example, it would be interesting to classify tightest gradings\footnote{A tight grading is called \emph{tightest} if every unit square inside the $m\times n$ rectangle either is below $\mathcal{P}(m,n)$, or intersects blue or red rectangles that correspond to the grading. For instance, the tight gradings in Example~\ref{tight_exmp}(2)(3)(4) are tightest.}. Another natural question is to find a bijective proof that shows that the coefficients of the wall-function in Figure~\ref{fig: D32} are equal to Catalan numbers.

\subsection{Rank 2 cluster scattering diagrams}
In \cite{ERS}, Elgin--Reading--Stella proposed several conjectures regarding rank 2 cluster scattering diagrams. We plan to explore these conjectures, using tight gradings.

\subsection{The case of higher rank}
We believe that there exists an analogue of greedy bases for (generalized) cluster algebras of higher rank. Once discovered, assuming that greedy bases and theta bases are equal, it would lead to a directly computable and combinatorial formula for broken lines bending only at a single wall, hence for cluster scattering diagrams of higher rank.

\bibliographystyle{amsplain}
\bibliography{bibliography.bib}

\end{document}